\input ./style/arxiv-general.cfg
\documentclass[aap,MSNbibl,seceqn,dvips]{arximspdf}
\makeatletter
   \@ifpackageloaded{graphicx}{}{\usepackage{graphicx}}
\makeatother

\doi{10.1214/15-AAP1138}
\volume{26}
\issue{4}
\pubyear{2016}
\firstpage{1996}
\lastpage{2029}
\docsubty{FLA}

\makeatletter
\newcommand{\rrvert}{\vert}
\newcommand{\llvert}{\vert}
\def\cal{\mathcal}

\newtheorem{theorem}{Theorem}[section]
\newtheorem{lemma}{Lemma}[section]
\newproclaim{example}{Example}[section]

\newtheorem{proposition}{Proposition}[section]
\newproclaim{remark}{Remark}
\newproclaim{definition}{Definition}[section]
\newproclaim{assumption}{Assumption}[section]
\newcommand{\cA}{{\mathcal A}}
\newcommand{\cF}{{\mathcal F}}
\newcommand{\cH}{{\mathcal H}}
\newcommand{\cP}{{\mathcal P}}
\newcommand{\bbE}{\mathbb{E}}
\newcommand{\eps}{\varepsilon}
\newcommand{\bbP}{\mathbb{P}}
\newcommand{\bbQ}{\mathbb{Q}}
\newcommand{\Om}{\Omega}
\newcommand{\bbR}{\mathbb{R}}
\newcommand{\what}{\widehat}
\def\rar{\rightarrow}

\newcommand{\chf}{\mathbf{1}}
\newcommand{\eid}{\stackrel{d}{=}}
\makeatother

\begin{document}
\begin{frontmatter}

\title{Markovian Nash equilibrium in financial markets with asymmetric
information and related forward--backward systems}
\runtitle{Asymmetric information and forward--backward systems}

\begin{aug}
\author[A]{\fnms{Umut} \snm{\c{C}etin}\corref{}\ead[label=e1]{u.cetin@lse.ac.uk}}
\and
\author[A]{\fnms{Albina} \snm{Danilova}\ead[label=e2]{a.danilova@lse.ac.uk}}
\runauthor{U. \c{C}etin and A. Danilova}
\affiliation{London School of Economics and Political Science}
\address[A]{Department of Statistics\\
London School of Economics \\
Columbia House\\
Houghton Street\\
London WC2A 2AE\\
United Kingdom\\
\printead{e1}\\
\phantom{E-mail:\ }\printead*{e2}}
\end{aug}

%
\received{\smonth{8} \syear{2014}}
%
\revised{\smonth{7} \syear{2015}}

%
\begin{abstract}
This paper develops a new methodology for studying continuous-time Nash
equilibrium in a financial market with asymmetrically informed agents.
This approach allows us to lift the restriction of risk neutrality
imposed on market makers by the current literature. It turns out that,
when the market makers are risk averse, the optimal strategies of the
agents are solutions of a forward--backward system of partial and
stochastic differential equations. In particular, the price set by the
market makers solves a nonstandard ``quadratic'' backward stochastic
differential equation. The main result of the paper is the existence of
a Markovian solution to this forward--backward system on an arbitrary
time interval, which is obtained via a fixed-point argument on the
space of absolutely continuous distribution functions. Moreover, the
equilibrium obtained in this paper is able to explain several stylized
facts which are not captured by the current asymmetric information models.
\end{abstract}

%

%
\end{frontmatter}

\section{Introduction}\label{sec1}
In this paper, we address the long-standing open problem\footnote{This
problem was posed by Subrahmanyam in \cite{sub}.} of existence of an
equilibrium in a financial market with asymmetrically informed traders
and risk averse market makers in continuous-time with finite horizon.
In such a market, the price of the traded asset is an equilibrium
outcome of a game between the market makers and an informed trader who
possesses superior information. Both market makers and the informed
trader choose their controls adapted to their filtrations. We assume
that the market makers obtain their information through their
interactions with the traders and have the obligation to absorb the
total demand for the asset. Therefore, their filtration is the one
generated by the total demand process, $Y$. The informed trader, on the
other hand, has the filtration jointly generated by the market prices
and her private information. In this game the market makers' control is
the price, $S$, while the control of the informed trader is her trading
strategy, $X$. Thus, the equilibrium price should satisfy the following
conditions: (i) the informed trader's optimisation problem has a
solution, and (ii) given this solution, the price $S$ fulfils the market
makers' objectives.\footnote{The focus of this paper is the equilibrium
between the market makers and the strategic informed trader as well as
the resulting price. We do not study in depth the interaction among the
market makers faced with a given demand process. The reader is referred
to the recent manuscript of Bank and Kramkov \cite{BK} for an in-depth
analysis of this interaction.}

The study of this game goes back to \cite{Kyle}, which is the canonical
model in market microstructure theory for the analysis of strategic
trading in the presence of private information (see \cite{Back,Brunnermeier}
and \cite{Ohara} for a review of Kyle's model as well as
a discussion of its relationship with other market microstructure
models). Various extensions of the original model have been studied in
the literature; see, among others,
\cite{B,bcw,BP,baruch,caldentey,cho,corcuera,dufresne} and \cite{HigaOrtiz}.

The original model and all these extensions assume that the market
makers are risk-neutral and compete in a \emph{Bertrand fashion} for the
total demand (see Section~12.C of \cite{M-C} for the definition). This
means that, in equilibrium, the utility of any market maker is a
martingale. Since the utility is linear, this in turn implies that the
optimal strategy for the market makers is to set the price to be the
conditional expectation of the fundamental value of the asset given
their filtration. In particular, in these models there is always a
unique price satisfying the objective of the market makers for any
control of the informed trader. Furthermore, the martingale property of
the price results in the optimal strategy of the informed trader being
\emph{inconspicuous} in the equilibrium; that is, the law of $Y$ in its
own filtration is the same as that of $Y-X$ in its own filtration.

Whereas the risk-neutrality of the market makers makes the model
tractable, it is not consistent with the observed market behaviour.
Indeed, there is a vast empirical evidence that the market makers are
risk averse and exercise their control in a way that total demand mean
reverts around a target level at a speed determined by their risk
aversion (see \cite{HS97} and \cite{MS93} for New York Stock Exchange,
\cite{HNV98} for London Stock Exchange, \cite{BR05} for Foreign
Exchange; for a survey of related literature and results, see Sections
1.2 and 1.3 in \cite{BGS05}).


Although relaxing the assumption of market makers' risk neutrality is
natural and has been prompted by the empirical evidence, there has been
only one attempt in the literature to investigate the effect of such an
extension. Subrahmanyam in \cite{sub} considered a one-period model
where market makers with identical exponential utilities set the price
that makes their utilities martingales. This assumption is the direct
analogue of the original Kyle model discussed above in the context of
risk averse market makers. The tractability of the model considered in
\cite{sub} relies on the fact that in a one-period setting there exists
an optimal response for the market makers for any strategy of the
insider. However, the existence of such responses is uncertain in a
multi-period setting. Indeed, Subrahmanyam noted that an extension of
his model to a multi-period setting is not possible due to the
strategic behaviour of the agents.

The aforementioned difficulty with the existence of an optimal response
for the market makers persists in continuous time. More precisely,
given a trading strategy of the informed trader, the optimal response
of the market makers is found via solving the backward stochastic
differential equation (BSDE)
%
\begin{eqnarray}
dS_t &=& Z_t \,d\beta_t -
\frac{c}{2}Y_t Z_t^2 \,dt, \label
{eq:s1:bsde}
\\
\exp(c Y_1 S_1 )&=&\bbE\bigl[\exp(c Y_1 V )
|\cF ^Y_1 \bigr], \label{eq:s1:bsdet}
\end{eqnarray}
where $c>0$ is a constant, $V$ is a bounded random variable
representing the fundamental value of the asset, $Y$ is a given total
demand process and $\beta$ is a Brownian motion with respect to $\cF^Y$
-- the filtration of the market makers generated by $Y$. A solution to
this BSDE is a pair $(Z,S)$ of $\cF^Y$-adapted processes satisfying
(\ref{eq:s1:bsde}) and (\ref{eq:s1:bsdet}). When this BSDE admits a
solution, $S$ is the price that makes the utilities of the market
makers martingales.

Although the terminal condition is unconventional, as $Y$ and $V$ are
given, the right-hand side of (\ref{eq:s1:bsdet}) is a fixed $\cF
^Y_1$-measurable random variable. Thus, we can rewrite the terminal
condition as $S_1=\xi$, which is bounded due to the boundedness of $V$.
The form of the driver, on the other hand, poses a real difficulty
since the process $Y$ multiplying $Z^2$ is in general unbounded. This
renders the system (\ref{eq:s1:bsde})--(\ref{eq:s1:bsdet}) outside the
realm of standard quadratic BSDEs.

The price response of the market makers is only one side of the
equilibrium. To characterise an equilibrium, we also need to find the
level of total demand, $Y$, implied by the informed agent's optimal
trading strategy. Consistent with the literature, we assume that the
total demand is driven by a Brownian motion and has a drift which is
determined by the informed trader. Hence, an equilibrium consists of
$(\alpha, S)$, where $\alpha$ is the optimal drift given $S$, and $S$
satisfies the forward--backward stochastic differential equation (FBSDE)
%
\begin{eqnarray}
dY_t&=&\,d\beta_t + \hat{\alpha}\bigl(t,(Y_s)_{s \leq t}
\bigr) \,dt, \label
{eq:s1:fbsdef}
\\
dS_t &=& Z_t \,d\beta_t -
\frac{c}{2}Y_t Z_t^2 \,dt, \label
{eq:s1:fbsdeb}
\\
\exp(c Y_1 S_1 )&=&\bbE\bigl[\exp(c Y_1 V )
|\cF ^Y_1 \bigr], \label{eq:s1:fbsdet}
\end{eqnarray}
where $\hat{\alpha}$ is the $\cF^Y$-optional projection of $\alpha
$. It
is well known that the existence of a solution for FBSDEs is quite
delicate even when the driver is globally Lipschitz and satisfies a
linear growth condition. Antonelli \cite{antonelli} showed the
existence and
uniqueness of a solution over a small time interval via a fixed-point
algorithm on a Banach space of processes. This result has been extended
by \cite{hamadene} and \cite{delarue} to arbitrary time intervals by
pasting solutions obtained for small time intervals. An alternative
technique for solving FBSDEs is the so-called \emph{four-step scheme}
introduced by \cite{maetal}, which requires strong smoothness on the
coefficients of the system and is based on the link between
quasi-linear partial differential equations. When the driver is
quadratic, the problem becomes more complicated and only few results
are available. Moreover, since available results originate from the
solvability of quadratic BSDEs, the standard assumption in the current
literature is that the driver is bounded by $k(1+z^2)$ for some
constant $k$ (see, e.g., \cite{imkelleretal}). However, as (\ref
{eq:s1:fbsdeb}) does not fit into the current paradigm of quadratic
BSDEs, these results are not applicable to our setting.

Despite these difficulties, we obtain a solution to this system with
$S_t=H(t,Y_t)$ for some smooth function $H$, when $\alpha$ is the
optimal drift of the informed trader given $S$. This solution provides
a Markovian equilibrium for the model that we consider. We show that in
this case the system (\ref{eq:s1:fbsdef})--(\ref{eq:s1:fbsdet})
transforms into
%
%
\begin{eqnarray}
H_t  + \frac{1}{2} H_{yy}&=&0, \label{eq:s1:NCH}
\\
dY_t &=&d\beta_t -\frac{c}{2}Y_t
H_y(t,Y_t)\,dt ,\label{eq:s1:NCY}
\\
V& \eid& H(1,Y_1), \label{eq:s1:NCF}
\end{eqnarray}
provided $Y$ has a smooth transition density, where the last equality
is an equality in distribution. This is still a forward--backward system
of a forward SDE and a backward PDE such that the terminal condition of
the PDE depends on the solution of the SDE, which in turn depends on
$H$. This coupling between the SDE and the PDE suggests a use of a
fixed-point algorithm.

Indeed, if we are given a continuous distribution for $Y_1$, (\ref
{eq:s1:NCF}) yields a function $H(1,y)$, which is increasing in $y$.
This allows us to obtain $H(t,y)$ via (\ref{eq:s1:NCH}), and $Y$ via
(\ref{eq:s1:NCY}). Hence, this procedure defines a mapping from the
space of distributions into itself. We show in Theorem~\ref{th:fp}, via
Schauder's fixed-point theorem, that this mapping has a fixed point
under the assumption that $V=f(\eta)$ for some increasing and bounded
$f$ satisfying some mild regularity conditions, and a standard normal
random variable $\eta$.

The validity of Schauder's fixed-point theorem in our setting relies
heavily on the properties of solutions of (\ref{eq:s1:NCY}) for any
given function $H$ satisfying (\ref{eq:s1:NCH}) with a bounded and
increasing terminal condition. These properties are explored in Lemmata
\ref{l:SD}--\ref{l:density}. In particular, we obtain a remarkable
connection between the laws of $Y_1$ and that of Brownian motion.
Namely, we prove that
\begin{eqnarray*}
\bbE\bigl[(Y_1-x)^+\bigr]&\geq& \bbE\bigl[\bigl(e^{-cC}B_{1}-x
\bigr)^+\bigr]>0 ,
\\
\bbE\bigl[(-x-Y_1)^+\bigr]&\geq&\bbE\bigl[\bigl(-x-e^{-cC}B_{1}
\bigr)^+\bigr]>0,
\end{eqnarray*}
for all $x>0$, where $C$ is a constant that depends only on the bound
on $H$. We also show that $Y$ has a smooth transition density.

The existence of solution to the system (\ref{eq:s1:NCH})--(\ref
{eq:s1:NCF}) ensures the existence of a Markovian solution for the
price process which makes the utilities of market makers martingales
once the drift of total demand, $Y$ has the form given in equation~(\ref
{eq:s1:NCY}). However, in order for such a drift to appear in
equilibrium, it should be optimal for the insider to choose a drift
whose $\cF^Y$-optional projection has this form.

To this end, we establish in Proposition~\ref{p:ioptimality} that the
sole criterion of optimality for the insider is that the strategy
fulfils the bridge condition $H(1, Y_1)=V$. Thus, if Markovian
equilibrium exists, the equilibrium pair $(H,Y)$ solves the system
(\ref
{eq:s1:NCH})--(\ref{eq:s1:NCF}) \emph{and} satisfies $H(1,Y_1)=V$. The
existence of such a pair is precisely the result of Theorem~\ref
{t:insider}, which allows us to establish the existence of the
equilibrium in Theorem~\ref{t:eq}.

The paper is structured as follows. Section~\ref{model} describes the
model we consider while Section~\ref{s:eq} is devoted to the (formal)
derivation of the system (\ref{eq:s1:NCH})--(\ref{eq:s1:NCF}) and
characterisation of the optimal strategy of the informed trader.
Section~\ref{s:main} establishes the existence of solution to the
system (\ref{eq:s1:NCH})--(\ref{eq:s1:NCF}) and Section~\ref{s:eqm}
proves the existence of the equilibrium. In Section~\ref{s:conclusion},
we discuss the impact of risk aversion on the market behaviour in the
equilibrium and explore the connections to the empirical literature.

\section{Market structure} \label{model}
Let $(\Omega, \cF, (\cF_t)_{t \in[0,1]}, \bbP)$ be a filtered
probability space satisfying the usual conditions of right continuity
and $\bbP$-completeness. We suppose that $\cF_0$ is not trivial and
there exists an $\cF_0$-measurable standard normal random variable,
$\eta$. Moreover, the filtered probability space also supports a
standard Brownian motion, $B$, with $B_0=0$, and thus, $B$ is
independent of $\eta$. We define $V:=f(\eta)$ for some bounded and
strictly increasing function $f$ with a continuous derivative.

As all the randomness in our model will depend only on $V$ and $B$, we
shall take $\cF=\sigma(\tilde{\mathcal{N}}, \tilde{\cF})$, where
$\tilde
{\cF}$ is the minimal $\sigma$-field with respect to which $V$ and
$(B_t)_{t \in[0,1]}$ are measurable and $ \tilde{\mathcal{N}}=\{E: E
\subset F \mbox{ for some } F \in\tilde{\cF} \mbox{ with } \bbP
(F)=0\}
$. Moreover, in view of the independence of $V$ and $B$, we may assume
the existence of a family of probability measures, $(\bbP^v)$ such that
the disintegration formula
\[
\bbP(E)=\int_{f(\bbR)} \bbP^v(E)\bbP(V \in dv)
\]
holds for all $E \in\cF$, and for all $v \in f(\bbR)$ the measure
$\bbP
^v$ satisfies $\bbP^v(E)=\bbP(E|V=v)$. The existence of such a family
is easily justified when we consider $\Om=f(\bbR) \times C([0,1],\bbR
)$, where $C([0,1],\bbR)$ is the space of real valued continuous
functions on $[0,1]$.

We consider a market in which the risk free interest rate is set to $0$
and a single risky asset is traded. The fundamental value of this asset
equals $V$, which will be announced at time $t=1$.

There are three types of agents that interact in this market:
\begin{longlist}[(iii)]
\item[(i)] Liquidity traders who trade for reasons exogenous to the
model and whose cumulative demand at time $t$ is given by $\sigma B_t$
for some constant $\sigma>0$.
\item[(ii)] A single informed trader, who knows $V$ from time $t=0$
onward, and is risk neutral. We will call the informed trader
\emph{insider} in what follows and denote her cumulative demand at
time $t$
by $X_t$. The filtration of the insider, $\cF^I$, is generated by
observing the price of the risky asset and $V$. Thus, an insider who
has the information that $V=v$ possesses the minimal right continuous
filtration generated by $V$ and the price process, and completed with
the null sets of $\bbP^v$.
\item[(iii)] Market makers observe only the net demand of the risky
asset, $Y=X+\sigma B$, thus, their filtration, $\cF^M$, is the minimal
right-continuous filtration generated by $Y$ and completed with $\bbP
$-null sets. The number of market makers is assumed to be $N \geq2$.

We also assume that the market makers have identical preferences
described by the common utility function, $U(x)=-e^{-\rho x}$, and
compete in a Bertrand fashion for the net demand of the risky asset. In
case of several market makers quoting the same winning price, we adopt
the convention that the total order is equally split among them.

Similar to \cite{B}, we assume that the market makers set the price of
risky security, $S$, as $S_t=H(t,Y_t)$ for some function $H$.
\end{longlist}
To understand the subtlety of the equilibrium derived later, it is
important to observe that an insider who is given the information that
$V=v$ has the probability measure $\bbP^v$ on $(\Om, \cF)$ while the
probability measure of the market makers is given by $\bbP$, and these
measures are singular with respect to each other as $\bbP^v(V=v)=1$,
whereas $\bbP(V=v)=0$ in our settings.

We now define admissibility of functions $H$ for the market makers
(which will be called \emph{pricing rule} in what follows) and
admissibility of the trading strategy of the insider. The conditions we
impose are standard in the literature and were first introduced in
\cite
{B}. The integrability conditions (\ref{a:Hintegrability}) and (\ref
{a:Xintegrability}) prevent the insider from following doubling
strategies (see \cite{B} for the discussion). The absolute continuity
of insider's strategies is without any loss of generality since
strategies with a martingale component and/or jumps are strictly
suboptimal as shown in \cite{B}.

%
\begin{definition} A function $H: \bbR_+\times\bbR\mapsto\bbR$ is a
\emph{pricing rule} if $H \in C^{1,2}$, strictly increasing in $y$ and
satisfies
%
\begin{equation}
\label{a:Hintegrability} \bbE H^2(1,\sigma B_1)<\infty\quad\mbox{and}\quad
\bbE\int_0^1 H^2(t,\sigma
B_t) \,dt < \infty.
\end{equation}
The class of such functions is denoted with $\cH$.
\end{definition}

Note that since any pricing rule is strictly monotone, $B$ is adapted
to $\cF^I$. The admissible strategies for the insider is defined in the
following.

\begin{definition} \label{d:admissible} An insider strategy, $X$, is
\emph{admissible} for a given pricing rule, $H$, if $X_t=\int_0^t\alpha
_s \,ds$ for some $\cF^I$-progressively measurable $\alpha$ such that,
for all $v \in f(\bbR)$, we have $\bbP^v(\int_0^1|\alpha_s| \,
ds<\infty)=1$,
%
\begin{equation}
\label{a:Xintegrability} \bbE^v\int_0^1
H^2(t,X_t+\sigma B_t) \,dt < \infty,
\end{equation}
and $\bbE^v [\min\{0,W^X_1\} ] >-\infty$, where $W^X_1$ is the
terminal wealth of insider given by
%
\begin{equation}\qquad
\label{d:iw} W_1^X:=\int_0^1
X_s \,dH(s,Y_s) + X_1\bigl(V-H(1,Y_1)
\bigr)=\int_0^1 \bigl(V-H(s,Y_s)
\bigr)\,dX_s.
\end{equation}
The class of admissible strategies for a given pricing rule $H$ will be
denoted by $\cA(H)$.
\end{definition}

Observe that for any $X$ of finite variation $W_1^X$ is well defined
since $V-H(s,X_s)$ is a continuous process for any pricing rule, $\bbP^v$-a.s.

The first term in (\ref{d:iw}) corresponds to continuous trading in the
risky asset, while the second term exists due to a potential
discontinuity in the asset price when the value becomes public
knowledge at time $t=1$. The second expression for the wealth follows
from integration by parts.

Given the definition of a pricing rule and admissible trading
strategies, we can now define an equilibrium as follows.
%

\begin{definition} A pair $(H^{\ast}, X^{\ast})$ is an equilibrium if
$H^{\ast} \in\cH$, $X^{\ast} \in\cA(H^{\ast})$, and
\begin{longlist}[(ii)]
\item[(i)] given $H^{\ast}$, the insider's strategy $X^{\ast}$ solves
her optimisation problem:
\[
\bbE^v \bigl[W^{X^{\ast}}_1 \bigr] = \sup
_{X \in\cA(H^{\ast})} \bbE^{v} \bigl[W^{X}_1
\bigr]\qquad \forall v \in f(\bbR).
\]
\item[(ii)] Given $X^{\ast}$, the pricing rule $H^{\ast}$ is such that
the market makers' wealth satisfies zero-utility gain condition, that
is, $U (G )$ is a $(\cF^M,\bbP)$-martingale, where
%
\begin{equation}
\label{e:mmG} G_t:=-\frac{1}{N}\int_0^t
Y^{\ast}_s \,dH^{\ast}\bigl(s,Y^{\ast}_s
\bigr) +\chf _{t=1}\frac{Y^{\ast}_1}{N}\bigl(H^{\ast}
\bigl(Y^{\ast}_1,1\bigr)-V\bigr).
\end{equation}
\end{longlist}
\end{definition}

The above is the formulation of a Markovian Nash equilibrium in our
model. The condition for the optimality of insider's strategy is a
straightforward description of the best response of the insider for a
given pricing rule. The market makers' optimality condition follows the
tradition of Kyle models where each market makers' utility remains a
martingale due to the Bertrand competition among them. Indeed, suppose
that one of the market makers, say $\mbox{MM}_i$, decides to deviate at
some time $t$ from this pricing rule by, for example, selling at a
higher price than $H$ would suggest in order to achieve a positive
utility gain. However, the other market makers could then offer to sell
at a slightly lower price which would still allow them to make a
positive utility gain. Moreover, as this lower price is more favourable
to the traders, no one will trade with $\mbox{MM}_i$ eliminating any
opportunities for a utility gain. Deviation from the zero-utility gain
condition by buying at a lower price is also suboptimal for a similar
reason. Clearly, buying (resp., selling) at a higher (resp., lower)
price is suboptimal since it leads to a loss in the utility. Thus, a
pricing rule satisfying the zero-utility gain condition is the best
response of the market makers. The zero-utility gain condition is also
a direct continuous-time analogue of the concept of \emph{autarky
utility} defining the equilibrium in the one-period Kyle model of \cite
{sub} studying the effects of the risk aversion of market makers on
equilibrium. Recall that the market makers are identical by assumption
and, therefore, they offer the same price quotes in equilibrium and the
order is split equally among them due to our order splitting convention
when there are more than one winning quote.
\section{Characterisation of equilibrium} \label{s:eq}
In this section, we show that a Markovian equilibrium of this game is
described by a forward--backward system of stochastic and partial
differential equations given by (\ref{eq:s1:NCH})--(\ref{eq:s1:NCF}) by
first studying the optimal response of the market makers for a given
strategy of the insider, and then characterising the profit maximising
strategies for the insider. The heuristic arguments below which are
used to characterise the equilibrium will be made rigorous in
subsequent sections.

Suppose that $X$ is an admissible trading strategy of the insider so
that $Y$ in its own filtration satisfies
\[
dY_t= \sigma\,dB^Y_t +\hat{
\alpha}_t \,dt,
\]
where $B^Y$ is an $\cF^M$-Brownian motion and $\hat{\alpha}$ is the
$\cF
^M$-optional projection of $\alpha$. The best response of the market
makers is to choose a price, $S$, that will satisfy the zero-utility
gain condition. Let price $S$ follow
\[
dS_t=Z_t \,dB^Y_t +
\mu_t \,dt,
\]
for some predictable process $Z$ and an optional process $\mu$ that are
to be determined by the market makers. As there is a potential
discrepancy between $S_1$ and $V$, there is a possibility of a jump in
the market makers' wealth at time $1$. More precisely,
\[
\Delta G_1= \frac{Y_1}{N}(S_1-V).
\]
However, the zero-utility gain condition implies
\[
1=\bbE \biggl[\exp \biggl(-\frac{\rho Y_1}{N}(S_1-V) \biggr) \Big|\cF
^M_1 \biggr],
\]
which is equivalent to
%
\begin{equation}
\label{e:bsdet} \bbE\biggl[\exp\biggl(\frac{\rho Y_1}{N}V \biggr)\Big |
\cF^M_1 \biggr]=\exp \biggl(\frac{\rho Y_1}{N}S_1
\biggr).
\end{equation}
On the other hand, if we compute the dynamics of $U(G)$ for $t<1$ by
It\^{o}'s formula, we obtain
\[
dU(G_t)=U(G_t)\frac{\rho}{N}Y_t \biggl\{
\sigma_t \,dB^Y_t + \biggl(\mu_t +
\frac{\rho}{2N} Y_t \sigma_t^2 \biggr)\,dt
\biggr\}.
\]
Reiterating the zero-utility gain condition for $t<1$ shows that we
must have
\[
\mu_t=-\frac{\rho}{2 N}Y_t Z_t^2.
\]
Therefore, the zero-utility gain condition stipulates that the price
$S$ follows
%
\begin{equation}
\label{e:bsde} dS_t =Z_t \,dB^Y_t
-\frac{\rho}{2 N}Y_t Z_t^2 \,dt,
\end{equation}
and the market makers' problem is to find $(Z, S)$ to solve (\ref
{e:bsde}) with the terminal condition (\ref{e:bsdet}) given the total
demand process $Y$.

The BSDE in (\ref{e:bsde}) is reminiscent of the quadratic BSDEs, which
have been studied extensively, and the connection of which to problems
arising in mathematical finance is well established (see, e.g., \cite
{kobylanski,briandhu,BE-K} and the references therein).
The essential deviation of (\ref{e:bsde}) from the BSDEs considered in
these papers is that the coefficient of $Z_t^2$ in (\ref{e:bsde}) is
$\frac{\rho}{2 N}Y_t$, which is in general unbounded. This makes the
direct application of the results contained in the current literature
for quadratic BSDEs to (\ref{e:bsde}) impossible.

However, if we turn to a Markovian equilibrium, that is, consider
$S_t=H(t,Y_t)$, it is natural to expect that in equilibrium $\hat
{\alpha
}_t=\hat{\alpha}_t(t,Y_t, S_t, Z_t)$ for some deterministic function
$\hat{\alpha}$ so that
%
\begin{equation}
\label{e:bsdef} dY_t= \sigma\,dB^Y_t +\hat{
\alpha}(t,Y_t, S_t, Z_t) \,dt.
\end{equation}
Thus, if a Markovian equilibrium can be attained it will provide a
Markovian solution to the FBSDE defined by (\ref{e:bsdet})--(\ref
{e:bsdef}), where $\hat{\alpha}$ is the optimal drift chosen by the insider.

We now turn to the optimisation problem for the insider when
$S_t=H(t,Y_t)$ for an admissible pricing rule $H$. Observe that from
the point of view of the insider the total demand process follows
\[
dY_t=\sigma\,dB_t + \alpha_t \,dt,
\]
for a given insider's strategy $X_t=\int_0^t\alpha_s \,ds$. And the value
function, $\Psi$, can be defined as
\[
\Psi(t,y)=\sup_{X \in\cA(H)}\bbE^{v} \biggl[\int
_t^1 \bigl(V-H(s,Y_s) \bigr)
\alpha_s \,ds \Big|Y_t=y \biggr].
\]
Then, a formal application of the dynamic programming principle leads
to the HJB equation
\[
\Psi_t +\frac{\sigma^2}{2}\Psi_{yy} +\sup
_{\alpha} \bigl\{\alpha(\Psi_y + V- H ) \bigr\}=0.
\]
Since the term to be maximised is linear in $\alpha$, the only way to
ensure the finiteness of solution is to set
\[
\Psi_y= H-V,
\]
which yields $\Psi_t +\frac{\sigma^2}{2}\Psi_{yy}=0$. Then, by
straightforward calculations we see that $H$ must satisfy a backward
heat equation
\[
H_t +\frac{\sigma^2}{2}H_{yy}=0,
\]
and, therefore, It\^{o}'s formula will yield that $S$ should satisfy
\[
dS_t=\sigma H_y(t,Y_t)\,dY_t.
\]
Combining this with (\ref{e:bsde}) and (\ref{e:bsdef}) implies
\[
\frac{z}{\sigma} \hat{\alpha}(t,y,s,z)=-\frac{\rho}{2 N}y z^2,
\]
that is,
%
\begin{equation}
\label{e:ahatopt} \hat{\alpha}(t,y,s,z)=-\frac{\rho\sigma}{2 N}y z
\end{equation}
as soon as we note that $z=\sigma H_y (t,y)$ by the choice of $S$.

The above form of $\hat{\alpha}$ is necessary in order for the market
makers to quote a Markovian pricing rule. However, in order for such
$\hat{\alpha}$ to appear in equilibrium, it should be optimal for the
insider to choose a drift whose $\cF^M$-optional projection has this
form. In Proposition~\ref{p:ioptimality}, we will show that the sole
criterion of optimality for the insider is that the strategy fulfils
the bridge condition $H(1, Y_1)=V$. Thus, if a Markovian equilibrium exists,
%
\begin{equation}
\label{e:s3:Yineq} dY_t= \sigma\,dB^Y_t -
\frac{\sigma^2 \rho}{2 N}Y_t H_y(t,Y_t),
\end{equation}
and $H$ solves the backward heat equation above and satisfies $H(1,Y_1)=V$.

As we show in Sections \ref{s:main} and \ref{s:eq} a pair $(H,Y)$
satisfying the above conditions exists for some admissible insider
trading strategy and that it indeed constitutes an equilibrium. In
order to see that this equilibrium is indeed feasible, suppose that we
have a pair $(H,Y)$ which solves the following system of equations:
%
%
\begin{eqnarray}
H_t  + \frac{1}{2}\sigma^2 H_{yy}&=&0,
\label{eq:NCH}
\\
dY_t &=&\sigma\,d\beta_t -\frac{\sigma^2 \rho}{2 N}Y_t
H_y(t,Y_t)\,dt, \label
{eq:NCY}
\\
V& \eid& H(1,Y_1), \label{eq:NCF}
\end{eqnarray}
with $Y_0=0$ where $\beta$ is a Brownian motion on some given
probability space and $Y$ is understood to be a strong solution of the
forward SDE. Further assume that the transition probability of $Y$
possesses a smooth density, $p$. Then the theory of filtration
enlargements gives us (see Theorem~1.6 in \cite{my}) that $Y$ solves
the SDE
%
\begin{equation}
\label{e:s3:Yenl} dY_t= \sigma\,d\tilde{\beta}_t + \biggl\{
\sigma^2\frac{p_y}{p}(t, Y_t; 1, Y_1)-
\frac{\sigma^2 \rho}{2 N}Y_t H_y(t,Y_t) \biggr\}\,dt,
\end{equation}
where $\tilde{\beta}$ is a Brownian motion with respect to the natural
filtration of $Y$ initially enlarged with the random variable $Y_1$
and, in particular, independent of $Y_1$. Thus, if $\tilde{V}$ is a
random variable with the same distribution as $V$ and independent of
$\tilde{\beta}$, we can replace $Y_1$ with $H^{-1}(1,\tilde{V})$ in
(\ref{e:s3:Yenl}) and obtain the SDE
\[
d\tilde{Y}_t= \sigma\,d\tilde{\beta}_t + \biggl\{
\sigma^2\frac{p_y}{p} \bigl(t, \tilde{Y}_t; 1,
H^{-1}(1,\tilde{V}) \bigr)-\frac{\sigma^2 \rho}{2 N}\tilde{Y}_t
H_y(t,\tilde{Y}_t) \biggr\}\,dt.
\]
Now, suppose that the solutions of this SDE are unique in law. Then
$\tilde{Y}$ will have the same law as $Y$, which yields in particular
that $\tilde{Y}_1=H^{-1}(1,\tilde{V})$ and in its own filtration
$\tilde
{Y}$ follows
\[
d\tilde{Y}_t =\sigma\,dB^{\tilde{Y}}_t -
\frac{\sigma^2 \rho}{2 N}\tilde{Y}_t H_y(t,\tilde{Y}_t)
\,dt,
\]
for some Brownian motion $B^{\tilde{Y}}$.

The above discussion makes it clear what the optimal strategy of the
insider should be given $H$. Since $V$ is independent of $B$, the
optimal number of shares of the risky asset held by the insider at time
$t$ equals
\[
\int_0^t \biggl\{ \sigma^2
\frac{p_y}{p} \bigl(s, {Y}_s; 1, H^{-1}(1,{V}) \bigr)-
\frac
{\sigma^2 \rho}{2 N}{Y}_s H_y(s,{Y}_s) \biggr
\}\,ds.
\]
This ensures that $Y$ follows (\ref{e:s3:Yineq}) in its own filtration
and $H(1,Y_1)=V$ achieving the optimality conditions for the insider as
well as those for the market makers.

These considerations imply that the question of existence of the
equilibrium can be reduced to the problem of existence of a solution to
the system (\ref{eq:NCH})--(\ref{eq:NCF}) with process $Y$ admitting a
smooth transition density. Despite the apparent simplicity, the
existence of a solution to this system is far from being a trivial
matter. Indeed, in order to determine $H$ via the basic PDE in (\ref
{eq:NCH}), we first need to know its boundary condition. However, the
boundary condition for $H$, (\ref{eq:NCF}), requires the knowledge of
the distribution of $Y_1$ which can only be determined if we know $H$.
Thus, this problem is appropriate for the employment of a fixed-point
theorem which indeed yields the existence of the solution as
demonstrated in the next section.


We end this section by proving the optimality criteria for the insider
that we used in order to establish the above system.
%

\begin{proposition} \label{p:ioptimality} Suppose $H$ is a pricing rule
satisfying
%
\begin{equation}
\label{eq:Hheat} H_t + \tfrac{1}{2}\sigma^2
H_{yy}=0.
\end{equation}
If $X_t=\int_0^t\alpha_s \,ds$ for some $\cF^I$-progressively measurable
$\alpha$ such that, for all $v \in f(\bbR)$, we have $\bbP^v(\int_0^1|\alpha_s| \,ds<\infty)=1$,
%
\begin{equation}
\label{eq:anona} \bbE^v \biggl(\int_0^1
H^2(t,X_t+\sigma B_t) \,dt \biggr) < \infty
\end{equation}
and
%
\begin{equation}
\label{eq:iopt} H(1, X_1+ Z_1)=V,\qquad \bbP^v\mbox{-a.s.},
\end{equation}
then $X \in\cA(H)$ and it is an optimal strategy for the insider.
\end{proposition}

\begin{pf}
We adapt the arguments in \cite{B} and \cite{Wu} to our case. Consider
the function
%
%
\begin{equation}
\label{def:generalG_a} \Psi(t,y):=\int_{\xi(t)}^y \bigl
\{H(t,u)-V \bigr\} \,du +\frac
{1}{2}\sigma^2 \int
_t^1 H_y \bigl(s,\xi(s) \bigr)\,ds,
\end{equation}
where $\xi(t)$ is the unique solution of $H(t,\xi(t))=V$.
Direct calculations show
%
%
\begin{equation}
\label{eq:G_a_x} \Psi_y(t,y)=H(t,y)-V
\end{equation}
and
%
%
\[
\Psi_t+\frac{\sigma^2}{2}\Psi_{yy}=0.
\]
Therefore, from (\ref{eq:G_a_x}) and It\^{o}'s formula it follows that
%
%
\begin{eqnarray}
\label{eq:opt_str_G_a} \Psi(1,Y_1)-\Psi(0,0)&=&\int_0^1
\bigl\{H(t,Y_t)-V \bigr\} \,dY_t
\nonumber
\\[-8pt]
\\[-8pt]
\nonumber
&=&-W^X_1
+ \int_0^1 \bigl\{ H(t,Y_t)-V
\bigr\} \sigma\,dB_t
\end{eqnarray}
for any $X$ such that $X_t=\int_0^t\alpha_s \,ds$ with $\bbP^v(\int_0^1|\alpha_s| \,ds<\infty)=1$.
Using (\ref{eq:opt_str_G_a}) and admissibility properties of $X$ (see
Definition~\ref{d:admissible}), insider's optimisation problem becomes
%
%
\begin{eqnarray}
\sup_{X\in{\cal{A}}(H)}\bbE^{v} \bigl[W^{X}_1
\bigr]&=& \sup_{X\in{\cal
{A}}(H)}\bbE^{v} \biggl[\int
_0 ^1 \bigl(V- H(t,Y_t) \bigr)
\,dX_t \biggr]
\\
&=& \bbE^{v} \bigl[\Psi(0,0) \bigr]-\inf_{X\in{\cal{A}}(H)}
\bbE^{v} \bigl[\Psi(1,Y_1) \bigr],
\end{eqnarray}
where the last equality is due to (\ref{a:Xintegrability}).

Since $\Psi(1,Y_1)=\int_{\xi(1)}^{Y_1} \{H(1,u)-V\}\,du$ is strictly
positive unless $Y_1=\xi(1)$ as $H(1,y)$ is strictly increasing, the
conclusion will follow as soon as $X$ is shown to be admissible. In
view of (\ref{eq:opt_str_G_a}),
\[
W^X_1= \Psi(0,0)+ \int_0^1
\bigl\{H(t,Y_t)-V \bigr\} \sigma\,dB_t,
\]
and, therefore, the admissibility of $X$ follows from (\ref{eq:anona}).
\end{pf}

\section{The main result and its proof}\label{s:main}
In this section, we state and prove the main result of this paper that
establishes the existence of a solution to the system given by (\ref
{eq:NCH})--(\ref{eq:NCF}).
%

\begin{theorem} \label{th:fp} There is a pair $(H,Y)$ that solves the
system of equations (\ref{eq:NCH})--(\ref{eq:NCF}). Moreover, $0<
H_y(t,y)\leq C \frac{1}{\sqrt{1-t}}$ for all $(t,y) \in[0,1) \times
\bbR$ and for some constant $C$. Furthermore, $Y$ is the unique strong
solution of (\ref{eq:NCY}) and admits a regular transition
density,\footnote{See the last paragraph on page 76 of \cite{MarRos} for
a definition.} $p(s,y;t,z)$, for all $0\leq s\leq t\leq1$ and $(y,z)
\in\bbR^2$ such that, for any fixed $(t,z)$, $p(s,y;t,z)>0$ on
$[0,t)\times\bbR$ and is $C^{1,2}([0,t) \times\bbR)$.
\end{theorem}

We will prove this theorem by an application of Schauder's fixed-point
theorem. Observe that if we start with an absolutely continuous
probability measure on $\bbR$ with full support, (\ref{eq:NCF}) yields
an increasing function $H(1,\cdot)$, which defines an $H$ solving
(\ref
{eq:NCH}). If we then plug this function into the SDE of (\ref
{eq:NCY}), we arrive at a new probability measure on $\bbR$ associated
with the distribution of $Y_1$. This procedure defines a transformation
from the space of probability measures on $\bbR$ into itself.
Application of Schauder's fixed-point theorem requires a suitable
choice of a closed and convex subset, $D$, of probability measures on
$\bbR$ such that the above transformation maps $D$ into itself and
satisfies the conditions of Schauder's fixed-point theorem.

Before we present the proof of the fixed-point result, we collect some
useful facts on the behaviour of the solutions of (\ref{eq:NCY}) in the
following lemmata. The first lemma observes a striking relationship
between the time 1 laws of the solutions of (\ref{eq:NCY}) and that of
$B_{\sigma^2}$. An immediate consequence of this lemma is that the law
of~$Y_1$, where $Y$ is the solution of (\ref{eq:NCY}) for a given $H$,
has a full support on $\bbR$. This property allows us to compute the
law of $Y_1$ via a Girsanov transform using the law of $B_1$, which is
achieved in the second lemma.
%

\begin{lemma} \label{l:SD} Suppose $H \in C^{1,2}([0,1)\times\bbR)$
satisfies $0\leq H_y(t,y)\leq C \frac{1}{\sqrt{1-t}}$ for all $(t,y)
\in[0,1) \times\bbR$, and some constant $C$. Let $c\geq0$ be a
constant, then the stochastic differential equation
%
\begin{equation}
\label{eq:genY} dY_t=\sigma\,dB_t- c Y_t
H_y(t,Y_t) \,dt
\end{equation}
has a unique strong solution on $[0,1]$. Moreover, for any $x>0$,
%
\begin{eqnarray}
\bbE\bigl[(Y_1-x)^+\bigr]&\geq& \bbE\bigl[\bigl(e^{-2cC}B_{\sigma^2}-x
\bigr)^+\bigr]>0, \label{e:USD1}
\\
\bbE\bigl[(-x-Y_1)^+\bigr]&\geq&\bbE\bigl[\bigl(-x-e^{-2cC}B_{\sigma^2}
\bigr)^+\bigr]>0 \label{e:USD2},
\end{eqnarray}
and, in particular, $\bbP(Y_1\leq y) \in(0,1)$ for all $y \in\bbR$.
\end{lemma}

%
\begin{pf}
Since $y H_y(t,y)$ is locally Lipschitz on $[0,T]\times\bbR$ for any
$T<1$, the above equation has a unique strong solution on $[0,T]$ upto
an explosion time~$\tau$. Since $T$ is arbitrary this implies the
existence of a unique continuous strong solution on $[0,1\wedge\tau)$.
Let $\tau_n:=\inf\{t\in[0,1): |Y_t|>n\}$ and observe that $\tau_n
\uparrow\tau$, a.s. Moreover, for any $t \in[0,1]$
\begin{eqnarray*}
Y_{t\wedge\tau_n}^2 &=& 2 \int_0^{t\wedge\tau_n}
Y_s \sigma\, dB_s - 2c \int_0^{t\wedge\tau_n}
Y_s^2 H_y(s,Y_s) \,ds +
\sigma^2 (t\wedge \tau _n)
\\
& \leq& 2 \int_0^{t\wedge\tau_n} Y_s \sigma
\,dB_s +\sigma^2 ({t\wedge\tau_n}).
\end{eqnarray*}
Thus, by It\^{o}'s isometry and the elementary inequality $x \leq1+ x^2$,
\[
\bbE \bigl[Y_{t\wedge\tau_n}^2 \bigr]\leq1+\sigma^2 +4
\sigma^2\int_0^t\bbE
\bigl[Y_s^2\chf_{[s <\tau_n]} \bigr] \,ds\leq1+
\sigma^2 +4 \sigma^2\int_0^t
\bbE \bigl[Y_{s\wedge\tau_n}^2 \bigr]\,ds.
\]
Therefore, Gronwall's inequality yields $\bbE[Y_{t\wedge\tau_n}^2]
\leq
( 1+\sigma^2 ) e^{4 \sigma^2}$ for all $t \in[0,1]$ and $n \geq1$.
Thus, $(Y_{t\wedge\tau_n})_{n \geq1}$ is uniformly integrable, and
consequently, $\bbP(\tau<t)=0$ and $\bbE[Y_{t}^2] \leq( 1+\sigma^2 )
e^{4 \sigma^2}$ for all $t \in[0,1]$, that is, $Y$ never explodes and
the SDE has a nonexploding strong solution.

To obtain the estimates (\ref{e:USD1}) and (\ref{e:USD2}), let
\[
\tilde{Y}_t =Y_t \exp \biggl(c \int_0^t
H_y (s,Y_s) \,ds \biggr)
\]
and observe that
\[
\tilde{Y}_t=\int_0^t \exp
\biggl(c \int_0^s H_y(r,Y_r)
\,dr \biggr)\sigma\,dB_s.
\]
Thus, $\tilde{Y}_t= W_{T_t}$ for some Brownian motion $W$ and the time
change $T_t$ satisfying
\[
\sigma^2 t \leq T_t= \sigma^2 \int
_0^t \exp \biggl(2 c \int_0^s
H_y(r,Y_r)\,dr \biggr)\,ds \leq\sigma^2
\exp(4 c C) t.
\]
Thus, by the optional sampling theorem, for any $K \in\bbR$ we have
\begin{eqnarray*}
\bbE\bigl[(\tilde{Y}_1-K)^+\bigr]&=&\bbE\bigl[(W_{T_1}-K)^+
\bigr] \geq\bbE\bigl[(W_{\sigma
^2}-K)^+\bigr]>0,
\\
\bbE\bigl[(K-\tilde{Y}_1)^+\bigr]&=&\bbE\bigl[(K-W_{T_1})^+
\bigr] \geq\bbE \bigl[(K-W_{\sigma^2})^+\bigr]>0,
\end{eqnarray*}
which implies (\ref{e:USD1}) and (\ref{e:USD2}).
\end{pf}
%
%

\begin{lemma} \label{l:girsanov} Let $h$ be a bounded, nondecreasing,
and absolutely continuous function, which is not constant. Consider the
solution, $H$, of (\ref{eq:NCH}) with the terminal condition $h$. Then
$|H(t,\cdot)|\leq\|h\|_{\infty}$ for $t\leq1$, and $0<H_y(t,\cdot)
\leq C\frac{1}{\sqrt{1-t}}$ for $t<1$, where $C= \sqrt{\frac
{2}{\sigma
^2 \pi}}\|h\|_{\infty}$. Consequently, there exists a unique, strong
solution, $Y$, of (\ref{eq:genY}) and, for any bounded and continuous
function $g$ and $T\leq1$, we have
\[
\bbE \bigl[g(Y_T) \bigr]=\bbE^{\bbQ} \bigl[g(\sigma
W_T) M_T \bigr],
\]
where $W$ is a Brownian motion on a filtered probability space $(\tilde
{\Om}, \tilde{\cF}, (\tilde{\cF}_t)_{t \in[0,1]},  \bbQ)$ and $(M_t)_{t
\in[0,1]}$ is a strictly positive $((\tilde{\cF}_t),\bbQ)$-martingale
given by
%
\begin{equation}
\label{e:rnd} M_t:=\exp\biggl(-c \int_0^t
W_s H_y(s,\sigma W_s)\,dW_s -
\frac{c^2}{2
}\int_0^t W_s^2
H^2_y(s,\sigma W_s) \,ds \biggr),
\end{equation}
with $c$ being the constant from Lemma~\ref{l:SD}. Furthermore, $\bbQ
$-a.s., $M_1 \leq e^{2c C}$, and
%
\begin{equation}
\label{e:Wdominates} \biggl\llvert\int_0^{\tau}
W_s H_y(s, \sigma W_s)\,dW_s
\biggr\rrvert\leq K \bigl(1+ |W_{\tau}| \bigr)\leq K \bigl(1+ W_{1}^*
\bigr),
\end{equation}
where $\tau$ is any stopping time with respect to the natural
filtration of $W$ such that $\tau\leq1$, $\bbQ$-a.s., $W^*_t=\sup_{s
\leq t}|W_s|$, and $K$ is some constant that depends only on $\sigma$
and $\|h\|_{\infty}$.
\end{lemma}

%
\begin{pf}
Observe that
\[
H(t,y)=\int_{\bbR} h(z) q \bigl(\sigma^2 (1-t),
z-y \bigr) \,dz,
\]
where $q(t, x)$ is the probability density of a normal random variable
with mean $0$ and variance $t$. Then, clearly, $|H(t,y)|\leq\int_{\bbR
} |h(z)| q(\sigma^2 (1-t), z-y) \,dz\leq\|h\|_{\infty}$. Moreover,
$H_y(t,y)$ is strictly positive whenever $t <1$. Indeed,
differentiating above, we have
\[
H_y(t,y)=\int_{\bbR}h(z) q_y \bigl(
\sigma^2(1-t),z-y \bigr) \,dz=\int_{\bbR} q \bigl(
\sigma^2(1-t),z-y \bigr)\,dh(z)>0.
\]

On the other hand,
%
\begin{eqnarray}
H_y (t,y)&=& \int_{\bbR} h(z)\frac{z-y}{\sigma^2(1-t)}q
\bigl(\sigma ^2(1-t),z-y\bigr) \,dz
\nonumber
\\[-8pt]
\\[-8pt]
\nonumber
&\leq&\sup_{z \in\bbR} h(z)\int_{\bbR}
\frac{|z-y|}{\sigma
^2(1-t)}q\bigl(\sigma^2(1-t),z-y\bigr) \,dz \leq C
\frac{1}{\sqrt{1-t}},
\end{eqnarray}
where $C= \|h(z)\|_{\infty} \sqrt{\frac{2}{\sigma^2 \pi}}$. Hence,
Lemma~\ref{l:SD} implies the existence and the uniqueness of a strong
solution to (\ref{eq:genY}).

Next, we will characterise the distribution of $Y$ on $[0,T]$ for $T<1$
by constructing a weak solution to (\ref{eq:genY}) via a Girsanov
transform. To this end, let $W$ be a Brownian motion on some filtered
probability space $(\tilde{\Om}, \tilde{\cF}, (\tilde{\cF}_t)_{t
\in
[0,1]},\break  \bbQ)$. Then $M$ is a martingale on $[0,T]$ by Corollary~3.5.16
in \cite{ks}. Thus, if we define $\tilde{\bbP}$ on $(\tilde{\Om},
\tilde
{\cF})$ by $\tilde{d\bbP}/d\bbQ=M_T$, $\sigma W$ solves (\ref{eq:genY})
under $\tilde{\bbP}$ on $[0,T]$. Due to the uniqueness in law of the
solutions of (\ref{eq:genY}), for any continuous and bounded function
$g$ we therefore have
\[
\bbE^{\tilde{\bbP}} \bigl[g(Y_T) \bigr]=E^{\bbQ} \bigl[g(
\sigma W_T) M_T \bigr].
\]
We next aim to extend the above equality to $T=1$, which would follow
from the dominated convergence theorem once we demonstrate that $M$ is
a bounded martingale. Direct calculations lead to
\begin{eqnarray*}
M_T&=&\exp \biggl(B(T,\sigma W_T)\\
&&{}+ c\int
_0^T \biggl\{H_y(s,\sigma
W_s) -\frac{1}{2}H_y(s,0)-\frac{c}{2}W_s^2
H_y^2(s,\sigma W_s) \biggr\}\,ds \biggr),
\end{eqnarray*}
where
%
\begin{equation}
\label{e:funcB} B(t,y)=-\frac{c}{\sigma^2}\int_0^y
x H_y(t,x) \,dx \leq0
\end{equation}
since $H_y$ is positive. Thus, for any $t \leq T$,
%
\begin{equation}
\label{e:Mbound} M_t \leq\exp\biggl(c\int_0^1
H_y(s,\sigma W_s) \,ds \biggr)\leq e^{2cC}
\end{equation}
implying
\[
\bbE^{\tilde{\bbP}} \bigl[g(Y_1) \bigr]=\bbE^{\bbQ} \bigl[g(
\sigma W_1) M_1 \bigr],
\]
where $M_1:=\lim_{T \rar1} M_T$.

Our next goal is to prove the estimate in (\ref{e:Wdominates}) which
will, in turn, imply that $M_1$ is strictly positive. Let $\tau$ be a
stopping time with respect to the natural filtration of $W$ and bounded
by $1$. Then
\begin{eqnarray*}
c \biggl\llvert\int_0^{\tau} W_s
H_y(s, \sigma W_s)\,dW_s\biggr\rrvert&
\leq&\bigl |B(\tau , \sigma W_{\tau})\bigr| + c\int_0^{\tau}
\biggl\llvert H_y(s, \sigma W_s) -\frac{1}{2}
H_y(s,0)\biggr\rrvert\,ds
\\
&\leq& \bigl|B(\tau, \sigma W_{\tau})\bigr| +3 cC,
\end{eqnarray*}
where $B(t,y)$ is given by (\ref{e:funcB}). A simple application of
integration by parts on $B(t,y)$ yields that $|B(t,y)| \leq\frac
{2c}{\sigma^2}|y|\|h\|_{\infty}$. Hence,
\[
\biggl\llvert\int_0^{\tau} W_s
H_y(s, \sigma W_s)\,dW_s \biggr\rrvert\leq
K \bigl(1+ |W_{\tau}| \bigr)\leq K \bigl(1+ W_{1}^* \bigr),
\]
for some $K$ that depends on $\sigma$ and $\|h\|_{\infty}$ only.

The above estimate also shows that $c \int_0^{t} W_s H_y(s, \sigma
W_s)\,dW_s$ is a square integrable martingale on $[0,1]$ with
\[
c^2 \int_0^1 E^{\bbQ}
\bigl[W_s^2 \bigl(H_y(s, \sigma
W_s) \bigr)^2 \bigr]\,ds \leq2 c^2
K^2 \bigl(1+ E^{\bbQ} \bigl(W^*_1
\bigr)^2 \bigr)<\infty.
\]
As $\{\omega: M_1(\omega)=0\}\subseteq\{\omega: \int_0^1
W_s^2(\omega)
(H_y(s, \sigma W_s(\omega)) )^2 \,ds=\infty\}$, this yields
that $M$ is strictly positive on $[0,1]$, $\bbQ$-a.s. and
\[
M_1=\exp \biggl(-c \int_0^1
W_s H_y(s,\sigma W_s)\,dW_s -
\frac{c^2}{2 }\int_0^1 W_s^2
H^2_y(s,\sigma W_s) \,ds \biggr).
\]
\upqed\end{pf}

Next lemma is not needed for the fixed-point algorithm that we will
consider in order to show the existence of a solution to the system
(\ref{eq:NCH})--(\ref{eq:NCF}). On the other hand, it shows that any
solution to (\ref{eq:NCY}) has a smooth transition density, which is
necessary to construct the equilibrium in our model.
%

\begin{lemma} \label{l:density} Let $h$ be a nonconstant, bounded,
nondecreasing, absolutely continuous function, the derivative of which
is bounded on compacts. Consider the solution, $H$, of (\ref{eq:NCH})
with the terminal condition $h$. Then the unique strong solution, $Y$,
of (\ref{eq:genY}) admits a regular transition density $p(s,y;t,z)$ for
all $0\leq s\leq t\leq1$ for $(y,z)\in\bbR^2$. Moreover, for any
fixed $(t,z)$, $p(s,y;t,z)>0$ on $[0,t)\times\bbR$ and is
$C^{1,2}([0,t) \times\bbR)$.
\end{lemma}

%
\begin{pf}
Due to the Lemma~\ref{l:girsanov}, we have $0<H_y(t,y) \leq C\frac
{1}{\sqrt{1-t}}$ for $t<1$ and $y \in\bbR$, where $H$ is the solution
of (\ref{eq:NCH}) with the terminal condition $h$, and $C= \sqrt
{\frac
{2}{\sigma^2 \pi}}\|h\|_{\infty}$. Furthermore, there exists a unique
the solution, $Y$, to (\ref{eq:NCY}) and for any bounded function $g$
and $0\leq t<u\leq1$,
\begin{eqnarray*}
&&\bbE \bigl[g(Y_u)|Y_t=y \bigr]\\
&&\qquad=\bbE^{\bbQ}
\biggl[g(\sigma W_u)\exp \biggl(-c \int_t^u
W_s H_y(s,\sigma W_s)\,dW_s\\
&&\qquad\quad{} -
\frac{c^2}{2 }\int_0^t W_s^2
H^2_y(s,\sigma W_s) \,ds \biggr) \Big|
W_t=\frac{y}{\sigma} \biggr].
\end{eqnarray*}
Thus, a regular transition density of $Y$ can be defined as
%
\begin{equation}
\label{e:tdrep} p(t,y;u,z)=q\bigl(\sigma^2(u-t),z-y\bigr)r(t,y;u,z),\qquad 0
\leq t<u \leq1,
\end{equation}
where
%
\begin{eqnarray}
\label{e:rdef} &&r(t,y;u,z)
\nonumber
\\
&&\qquad:=\bbE^{\bbQ^{y \rar z}_{t \rar u}} \biggl[\exp\biggl(-
\frac
{c}{\sigma^2}\int_t^u Y_s
H_y(s,Y_s)\,dY_s\\
&&\qquad\quad{} -\frac{c^2}{2\sigma
^2}\int
_t^u Y_s^2
H^2_y(s, Y_s) \,ds \biggr) \biggr],\nonumber
\end{eqnarray}
with $\frac{Y}{\sigma}$ being a Brownian bridge from $\frac
{y}{\sigma}$
to $\frac{z}{\sigma}$ on the interval $[t,u]$ under measure $\bbQ^{y
\rar z}_{t \rar u}$. Indeed, the representation (\ref{e:tdrep}) holds
once we show that $r$ is a measurable function and Chapman--Kolmogorov
equations hold. In fact, as we show below $r$ is continuous with
respect to all its parameters, hence, measurable (the easy task of
validating Chapman--Kolmogorov equation is left to the reader).

First, observe that the It\^{o} formula and the PDE (\ref{eq:NCH})
satisfied by $H$ yield [recall that $B(t,y)$ is given by (\ref{e:funcB})]
%
%
\begin{eqnarray}\label{eq:ryz}
&&e^{B(t,y)-B(u,z)}r(t,y;u,z)
\nonumber
\\
&&\qquad=\bbE^{\bbQ^{y \rar z}_{t \rar u}} \biggl[\exp \biggl(c\int_t^u
\biggl\{ H_y(s,Y_s) -\frac{1}{2}H_y(s,0)-
\frac{c}{2\sigma^2}Y_s^2 H_y^2(s,Y_s)
\biggr\} \,ds \biggr) \biggr]
\nonumber
\\[-8pt]
\\[-8pt]
\nonumber
&&\qquad=\exp \biggl(-\frac{c}{2}\int_t^u
H_y(s,0)\,ds \biggr)\\
&&\qquad\quad{}\times\bbE^{\bbQ^{y \rar
z}_{t \rar u}} \biggl[\exp \biggl(c\int
_t^u \biggl\{H_y(s,Y_s)
-\frac
{c}{2\sigma^2}Y_s^2 H_y^2(s,Y_s)
\biggr\}\,ds \biggr) \biggr].
\nonumber
\end{eqnarray}
Moreover, in view of the SDE representation of Brownian bridges (see
Section~5.6.B in \cite{ks}), the law of $Y$ under $\bbQ^{y \rar z}_{t
\rar u}$ is the same as that of $\tilde{Y}$ under $\bbQ$, where
\[
\tilde{Y}_s:=y\frac{u-s}{u-t}+ z \frac{s-t}{u-t}+ \sigma(u-s)
\int_t^s\frac{dW_r}{u-r},\qquad s \in[t,u].
\]
Therefore,
\begin{eqnarray*}
&&\bbE^{\bbQ^{y \rar z}_{t \rar u}} \biggl[\exp \biggl(c\int_t^u
\biggl\{ H_y(s,Y_s) -\frac{c}{2\sigma^2}Y_s^2
H_y^2(s,Y_s) \biggr\}\,ds \biggr) \biggr]
\\
&&\qquad=\bbE^{\bbQ} \biggl[\exp \biggl(c\int_t^u
\biggl\{H_y(s,\tilde{Y}_s) -\frac
{c}{2\sigma^2}
\tilde{Y}_s^2 H_y^2(s,
\tilde{Y}_s) \biggr\}\,ds \biggr) \biggr],
\end{eqnarray*}
and the desired continuity follows from the continuity of $\tilde{Y}$
with respect to $(t,y,u,z)$ and the dominated convergence theorem that
applies due to the bounds on $H_y$.

In order to show that $r(t,y;u,z)>0$ for all $u \leq1$, it suffices to
show that
%
\begin{equation}
\bbQ^{y \rar z}_{t \rar u} \biggl( \int_t^u
Y_s^2 H_y^2(s,Y_s)
\,ds <\infty \biggr)=1. \label{e:rispos}
\end{equation}
Indeed, due the uniform bounds on $H_y$, the nonnegative random
variable inside the conditional expectation in (\ref{eq:ryz}) is zero
only if $\int_t^u Y_s^2 H_y^2(s,Y_s)\,ds=\infty$. To this end, fix an
$\omega
$ and observe that $K_{t,u}:=\sup_{t \leq s \leq u} Y_s^2$ satisfies
$\bbQ^{y \rar z}_{t \rar u}(0<K_{t,u}<\infty)=1$. Therefore,
\begin{eqnarray*}
\frac{1}{4 K_{t,u}}\int_t^u
Y_s^2 H_y^2(s,Y_s)
\,ds &\leq& \frac
{1}{4}\int_t^u
H_y^2(s,Y_s)\,ds
\\
&=& \frac{1}{4} \int_t^u \biggl(\int
_{\bbR} h'(z) q \bigl(\sigma^2(1-s),
z-Y_s\bigr)\,dz \biggr)^2\,ds
\\
&\leq& \int_t^u \biggl(\int
_{-1}^1 h'(z+Y_s) q
\bigl(\sigma^2(1-s), z\bigr)\,dz \biggr)^2\,ds
\\
&&+ \int_t^u \biggl(\int_{1}^{\infty}
h'(z+Y_s) q \bigl(\sigma^2(1-s), z\bigr)
\,dz \biggr)^2\,ds
\\
&&+ \int_t^u \biggl(\int_{-\infty}^{-1}
h'(z+Y_s) q \bigl(\sigma^2(1-s), z\bigr)
\,dz \biggr)^2\,ds.
\end{eqnarray*}
Observe that, for the fixed $\omega$, $Y$ is a continuous function of time
and, therefore, takes values in a compact set, which implies that $
h'(z+Y_s)$ is bounded for $z \in[-1,1]$ and all $s \in[t,u]$. This
implies that the first integral is finite since $\int_{-1}^1 q (\sigma
^2(1-s), z)\,dz <1$.

To see the finiteness of the second integral, apply integration by
parts to get
\[
\int_t^u \biggl(-\tilde{h}(1+Y_s)q
\bigl(\sigma^2(1-s),1 \bigr) +\int_1^{\infty}
\tilde{h}(z+Y_s)\frac{z}{\sigma^2(1-s)}q \bigl(\sigma^2(1-s),z
\bigr)\,dz \biggr)^2\,ds,
\]
where $\tilde{h}=h+\|h\|_{\infty}$. Note that $\tilde{h}$ is positive,
therefore, the above integral is bounded from above by
\[
8\|h\|_{\infty}^2\int_t^u
q^2 \bigl(\sigma^2(1-s),1 \bigr)\,ds <\infty.
\]
The third integral can be shown to be finite in the same way.

In order to show that $p(t,y; u,z) \in C^{1,2}([0,u) \times\bbR)$ for
fixed $(u,z)$, where $u < 1$, we will show that it is the fundamental
solution of a parabolic differential equation (see page 3 of \cite
{friedman} for the definition of fundamental solutions). In view of the
relationship between the fundamental solutions of PDEs and transition
densities of diffusion processes (see the discussion following
Definition~5.7.9 in \cite{ks}), let us consider the PDE
%
\begin{equation}
\label{eq:pdeY} u_t+\tfrac{1}{2}\sigma^2
u_{yy} -cy H_y u_y=0
\end{equation}
on the interval $[0,T]$ where $T<1$. The existence of a fundamental
solution to this PDE will follow from Theorem~1 in \cite{besala} once
we show that conditions (i)--(iii) on page 28 of \cite{besala} are
satisfied. Condition (i) is trivially satisfied for $\sigma$ being a
constant. Moreover, since
\[
\biggl\llvert\frac{\partial}{\partial y} y H_y \biggr\rrvert =\llvert
H_y + y H_{yy}\rrvert
\]
and $H_y(t,y) \leq C \frac{1}{\sqrt{1-T}}$ for all $(t,y) \in
[0,T]\times\bbR$, we can conclude that the function $\frac{\partial
}{\partial y} y H_y$ is locally bounded if $H_{yy}$ can be shown to be
bounded in $(t,y)$ when $y$ belongs to a bounded interval. Indeed, by
directly differentiating $H$ we obtain
\begin{eqnarray*}
|H_{yy}|&\leq& \frac{1}{\sigma^2(1-t)} \biggl(\int_{\bbR}
\bigl|H(1,z)\bigr|q\bigl(\sigma ^2(1-t),z-y\bigr) \,dz\\
&&{} +\int
_{\bbR} \bigl|H(1,z)\bigr|\frac{(z-y)^2}{\sigma
^2(1-t)}q\bigl(\sigma^2(1-t),z-y
\bigr) \,dz \biggr)
\\
&\leq& 2 \frac{\|h(z)\|_{\infty}}{\sigma^2(1-T)},
\end{eqnarray*}
that is, $H_{yy}$ is uniformly bounded on $[0,T]\times\bbR$. Thus, we
have shown that condition (ii) was satisfied. Since the constant
functions solve the (\ref{eq:pdeY}), condition (iii) is satisfied as
well; thus, a fundamental solution, $\Gamma(t,y;s,z)$ to (\ref
{eq:pdeY}) exists. In particular, if one considers this PDE with the
boundary condition $u(T,y)=g(y)$ for some bounded $g$, the solution is
given by
\[
u(t,y)=\int_{\bbR}g(z)\Gamma(t,y; T,z) \,dz.
\]
On the other hand, since the SDE (\ref{eq:genY}) satisfies the
hypotheses of Theorem~5.7.6 in \cite{ks} on the time interval $[0,T]$,
$u$ has the following stochastic representation by this theorem:
\[
u(t,y)=\bbE \bigl[g(Y_T)|Y_t=y \bigr]=\int
_{\bbR} g(z)p(t,y;T,z) \,dz.
\]
Thus,
\[
\int_{\bbR} g(z)p(t,y;T,z) \,dz=\int_{\bbR}g(z)
\Gamma(t,y; T,z) \,dz,
\]
and since $g$ is arbitrary and both $\Gamma$ and $p$ are continuous in
their parameters, we deduce $p(t,y;T,z)=\Gamma(t,y;T,z)$ for all $0
\leq t<T<1$, and thus, it is $C^{1,2}$ on $[0,T)\times\bbR$ for $T<1$.

To show that $p(t,y; 1, z)$ is $C^{1,2}$ on $[0,1)\times\bbR$ for each
$z$ consider (\ref{eq:pdeY}) with the boundary condition
$u(T,y)=p(T,y;1,z)$. Note that $u(T,y)$ is bounded in $y$ since,
due to (\ref{eq:ryz}), we have
%
\begin{eqnarray}\label{eq:ryzbd}
r(t,y;u,z) &\leq& e^{2cC} e^{-B(t,y)}= e^{2cC}e^{({c}/{\sigma^2})
(yH(t,y)-\int_0^yH(t,x) \,dx)}
\nonumber
\\[-8pt]
\\[-8pt]
\nonumber
&\leq& e^{2cC} e^{({c\|h\|_{\infty
}}/{\sigma^2})|y|},
\end{eqnarray}
where the first inequality is due to bounds on $H_y(t,y)$ and the last
one due to the bounds on $H(t,y)$ obtained in Lemma~\ref{l:girsanov}.
Thus, there exists a unique classical solution, $u(t,y)$, to (\ref
{eq:pdeY}), with the boundary condition $u(T,y)=p(T,y;1,z)$, given by
\[
u(t,y)=\int_{\bbR} p(t,y;T,x)p(T,x;1,z) \,dx
\]
by the definition of fundamental solutions. However, by
Chapman--Kolmogorov equations,
%
\begin{equation}
\label{e:pt1yz} \int_{\bbR} p(t,y;T,x)p(T,x;1,z)
\,dx=p(t,y;1,z),
\end{equation}
which in turn yields that $p(t,y;1,z)\in C^{1,2}([0,T)\times\bbR)$.
Since $T$ is arbitrary, we have $p(t,y;1,z)\in C^{1,2}([0,1)\times\bbR)$.
\end{pf}
%

Having collected all the prerequisites, we can now prove our main
theorem.

\begin{pf*}{Proof of Theorem~\ref{th:fp}}
In the setting of Lemma~\ref{l:girsanov}, $M$ defines an equivalent
change of measure between the laws of $Y$ and $\sigma W$. Thus, if we
define $r(y)$ by [see (\ref{e:rdef})]
%
\begin{equation}
\label{d:r} r(y):=r(0,0;1,y)=\bbE^{\bbQ} \biggl[M_1
\Big|W_1=\frac{y}{\sigma} \biggr],
\end{equation}
then
\[
\bbE \bigl[g(Y_1) \bigr]=\int_{\bbR} g(y) q \bigl(
\sigma^2,y \bigr)r(y) \,dy
\]
and, therefore, the probability density of $Y_1$ under $\bbP$ is given by
%
\begin{equation}
\label{e:Ydensity} q\bigl(\sigma^2,y\bigr)r(y) \leq q\bigl(
\sigma^2,y\bigr) e^{2cC}.
\end{equation}
The existence of a solution to the system of equations (\ref
{eq:NCH})--(\ref{eq:NCF}) will be shown via a fixed-point argument
applied to a certain operator mapping a class of distribution functions
on $\bbR$ into itself.

Schauder's fixed-point theorem (see Theorem~7.1.2 in \cite{friedman})
states that if $D$ is a closed convex subset of a Banach space and
$T:D\mapsto D$ is a continuous operator, then it has a fixed point if
the space $T D$ is precompact, that is, every sequence in $T D $ has a
subsequence which converges to some element of the Banach space. In
order to apply this theorem, we first need to find a suitable Banach
space which contains a class of probability distribution functions on
$\bbR$ that is large enough to contain the distribution of $Y_1$ where
$Y$ is one of the components of the solution to the system of equations
(\ref{eq:NCH})--(\ref{eq:NCF}). In view of the above discussion, the
distribution of $Y_1$ will be continuous, in fact it will admit a
density. Thus, we may take $C_b(\bbR)$, \emph{the space of bounded
continuous functions on $\bbR$}, equipped with the sup norm as our
underlying Banach space and set $\cP$ as the space of absolutely
continuous distribution functions on $\bbR$, that is, $P \in\cP$ if
$P$ is increasing, $P(-\infty)=0$, $P(\infty)=1$, and there exists a
measurable function $P'$ such that $P(y)=\int_{-\infty}^y P'(z) \,dz$.
Then we can define the set
\begin{eqnarray*}
D&=& \biggl\{P \in\cP: P'(z) \leq
C^{\ast} q \bigl(\sigma^2, z \bigr), \forall z \in\bbR,
\\
&&\int_x^{\infty} (y-x)P'(y)\,dy \geq
\bbE \biggl[ \biggl(\frac{1}{C^{\ast}} W_{\sigma^2}-x \biggr)^+ \biggr],
\forall x>0,
\\
&&{} -\int_{-\infty}^{-x} (y+x) P'(y)\,dy
\geq\bbE \biggl[ \biggl(-x-\frac
{1}{C^{\ast}}W_{\sigma^2} \biggr)^+ \biggr],
\forall x>0 \biggr\},
\end{eqnarray*}
where
\[
C^{\ast}:=\exp \biggl(\frac{\rho\| f(z)\|_{\infty}}{N} \sqrt {\frac{2}{\pi
}}
\sigma \biggr).
\]
The reason for this judicious choice of $C^{\ast}$ will become apparent
when we define the operator $T$. We will prove the existence of a fixed
point in four steps.

\textit{Step} 1: \emph{$D$ is a closed convex set.} It is clear that $D$
is convex. To see it is also closed, suppose that $P_n$ is a sequence
of elements in $D$ converging to some element, $P$, of the Banach space
in the sup norm. Clearly, $P$ is increasing with $P(-\infty)=0$ and
$P(\infty)=1$. Moreover, for any $x \leq y$ in $\bbR$, it follows from
Fatou's lemma that
\[
0 \leq P(y)- P(x)=\lim_{n\rar\infty} \int_x^y
P_n'(z) \,dz \leq\int_x^y
\limsup_{n \rar\infty} P_n'(z) \,dz,
\]
since each $P_n'$ is bounded from above by the same integrable
function, which in turn is an upper bound to the positive function
$\limsup_{n \rar\infty} P_n'$. However, this implies that $P$ is
absolutely continuous and, in particular, there exists a function $P'$
with $0 \leq P'(z) \leq\limsup_{n \rar\infty} P_n'(z) \leq C^{\ast}
q(\sigma^2, z)$ for all $z \in\bbR$.

To complete the proof that $D$ is closed, we need to show
\[
\int_x^{\infty} (y-x)P'(y)\,dy \geq
\bbE^{\bbQ} \biggl[ \biggl(\frac
{1}{C^{\ast}} W_{\sigma^2}-x \biggr)^+
\biggr] \qquad\forall x>0.
\]
Since $P_n$ converges $P$ weakly, there exists a probability space
supporting random variables $(Y_n)_{n \geq0}$ and $Y$ such that $Y_n
\rar Y$, a.s., $Y_n$ has distribution $P_n$, and $Y$ has distribution
$P$. Note that one can directly verify that
\[
\int_{\bbR}(y-x)^2 P_n(dy) \leq
C^{\ast} \bbE \bigl[(W_{\sigma^2}-x)^2 \bigr],
\]
which shows the uniform integrability of the sequence $(Y_n-x)^+$. Therefore,
\[
\int_x^{\infty} (y-x)P(dy)=\lim
_{n \rar\infty}\int_x^{\infty}
(y-x)P_n(dy)\geq\bbE \biggl[ \biggl(\frac{1}{C^{\ast}}
W_{\sigma^2}-x \biggr)^+ \biggr].
\]
Similar arguments show the other inequality. Thus, $D$ is closed.

\textit{Step} 2: \emph{Defining the operator $T$.} For any $P \in D$,
let $H: [0,1]\times\bbR\mapsto\bbR$ be the unique function which
solves the following boundary value problem:
%
\begin{eqnarray} \label{eq:BVPH}
H_t +\frac{\sigma^2}{2}H_{yy}&=&0,
\nonumber
\\[-8pt]
\\[-8pt]
\nonumber
H(1,y)&=&f\bigl(\Phi^{-1}\bigl(P(y)\bigr)\bigr),
\end{eqnarray}
where $\Phi$ is the cumulative distribution function of a standard
normal random variable. Observe that $h(z):=f(\Phi^{-1}(P(z)))$ is a
bounded, increasing function. Moreover, its derivative given by
$f'(\Phi
^{-1}(P(y)))(\Phi^{-1})'(P(y))P'(y)$ is well defined for all $y\in
\bbR$
as $P(y)\in(0,1)$ for all $P\in D$ and $y\in\bbR$ and, therefore, $h$
is also absolutely continuous. Thus, by Lemma~\ref{l:girsanov}, for all
$t<1$, $0<H_y(t,y)\leq C \frac{1}{\sqrt{1-t}}$, where $C= \sqrt
{\frac
{2}{\sigma^2 \pi}} \|f\|_{\infty} $ is independent of the choice of $P$.

To this $H$ one can associate a unique process $Y$ which solves (\ref
{eq:genY}) for $c=\frac{\sigma^2 \rho}{2N}$ and $Y_1$ is a continuous
random variable with the probability density $q(\sigma^2,y) r(y)$,
where $r$ is defined in (\ref{d:r}). Thus, we can define
\[
T P(y)=\int_{-\infty}^y q \bigl(
\sigma^2,z \bigr) r(z) \,dz.
\]
Note that $TP$ belongs to $D$ due to (\ref{e:USD1}), (\ref{e:USD2}) and
(\ref{e:Ydensity}).

\textit{Step} 3: \emph{$T$ is precompact.}
Since $TD$ is an equicontinuous family of functions, by a version of
the Ascoli--Arzela theorem (see Corollary III.3.3 in \cite{lang}), if
$P_n$ is a sequence in $TD$ then it admits a subsequence which
converges pointwise to $P \in C_b(\bbR)$. Moreover, this convergence is
uniform on every compact subset of $\bbR$. This would mean that $TD$ is
precompact once we show that the convergence is uniform over all $\bbR$.

To do so, let us assume without loss of generality that $P_n$ itself is
the convergent subsequence and consider any $\eps>0$. Due to the
definition of $D$, there exist $x^{\ast}$ and $x_{\ast}$ such that
\begin{eqnarray*}
P_n(x) &\leq& C^{\ast} \int_{-\infty}^x
q \bigl(\sigma^2,y\bigr) \,dy \leq C^{\ast
} \int
_{-\infty}^{x_{\ast}} q \bigl(\sigma^2,y\bigr)
\,dy
\\
&=&C^{\ast} \Phi(x_{\ast}) \leq\frac{\eps}{6}\qquad \forall x
\leq x_{\ast};
\\
1-P_n(x) &\leq& C^{\ast} \int^{\infty}_x
q \bigl(\sigma^2,y\bigr) \,dy \leq C^{\ast} \int
^{\infty}_{x^{\ast}} q \bigl(\sigma^2,y\bigr)
\,dy
\\
&=&C^{\ast} \Phi\bigl(-x^{\ast}\bigr) \leq\frac{\eps}{6}\qquad
\forall x \geq x^{\ast}.
\end{eqnarray*}
Since $P_n$ converges to $P$ pointwise, we also have with the same
$x^{\ast}$ and $x_{\ast}$ that
\[
P(x) \leq\frac{\eps}{4} \qquad\forall x \leq x_{\ast};\qquad 1-P(x) \leq
\frac{\eps}{4}\qquad \forall x \geq x^{\ast}.
\]
Since the convergence is uniform on the compact $[x_{\ast},x^{\ast}]$,
there exist a $K$ such that for all $n \geq K$
\[
\sup_{x \in[x_{\ast},x^{\ast}]}\bigl|P_n(x)-P(x)\bigr| \leq\frac{\eps}{3}.
\]
Thus, for any $n \geq K$ we have
\begin{eqnarray*}
\sup_{x \in\bbR}\bigl|P_n(x)-P(x)\bigr| & \leq& \sup
_{x \in[x_{\ast
},x^{\ast
}]}\bigl|P_n(x)-P(x)\bigr| + \sup_{x \in(-\infty, x_{\ast}]}
\bigl(P_n(x)+P(x)\bigr)
\\
&&{}+ \sup_{x \in[x^{\ast}, \infty)}\bigl(1-P_n(x)+1-P(x)\bigr) \leq
\eps.
\end{eqnarray*}
Thus, we have shown that the convergence of $P_n$ to $P$ is uniform on
$\bbR$, that is, $TD$ is precompact in $C_b(\bbR)$ equipped with the
sup norm. Hence, Schauder's fixed-point theorem yields $T$ has a fixed
point provided $T$ is a continuous operator, which we show next.

\textit{Step} 4: \emph{$T$ is continuous.}
To this end, let $(P_n)_{n \geq1} \subset D$ converge to $P \in D$ in
the sup norm. As $TP_n$ and $TP$ belong to $D$, in view of Problem 14.8(c) in \cite{billingsley}, pointwise convergence of $TP_n$ to $TP$ will
imply uniform convergence since $TP$ is continuous. To each $P_n$ and
$P$, we can associate functions $H^n$ and $H$, $B^n$ and $B$ [see (\ref
{e:funcB})], and the processes $M^n$ and $M$ from Lemma~\ref{l:girsanov}.

Pointwise convergence of $TP_n$ to $TP$ will follow immediately once we
can show that for any continuous and bounded function $g$
\[
\lim_{n \rar\infty} \bbE^{\bbQ} \bigl[g(\sigma
W_1) M^n_1 \bigr] = \bbE^{\bbQ
}
\bigl[g(\sigma W_1) M_1 \bigr].
\]
In view of the uniform bound on $M^n$ and $M$ due to (\ref{e:Mbound}),
the above convergence will hold if we can show that $M^n_1$ converges
to $M_1$ in $Q$-probability.

In order to get the estimates to prove this convergence, first note
that, due to Lemma~\ref{l:girsanov} for any stopping time, $\tau$,
bounded by $1$, we have
\[
c \biggl\llvert\int_0^{\tau} W_s
H^n_y(s, \sigma W_s)\,dW_s
\biggr\rrvert\leq K \bigl(1+ |W_{\tau}| \bigr)\leq K \bigl(1+ W_{1}^*
\bigr),
\]
for some $K$ independent of $n$. This shows that $c \int_0^{t} W_s
H^n_y(s, \sigma W_s)\,dW_s$ is a square integrable martingale on
$[0,1]$ with
\[
c^2 \int_0^1 \bbE^{\bbQ}
\bigl[W_s^2 \bigl(H^n_y(s,
\sigma W_s) \bigr)^2 \bigr]\,ds \leq K \bigl(1+
\bbE^{\bbQ} \bigl(W^*_1 \bigr)^2 \bigr),
\]
where $K$ is a constant independent of $n$. Let
\[
N^n_t:=\int_0^{t}
W_s \bigl\{H^n_y(s, \sigma
W_s)-H_y(s, \sigma W_s) \bigr\}
\,dW_s.
\]
Since [recall that $B(t,y)$ is given by (\ref{e:funcB})]
\[
-c \int_0^{1} W_s
H_y(s, \sigma W_s)\,dW_s=B(1, \sigma
W_{1}) + c\int_0^{1} \biggl
\{H_y(s, \sigma W_s) -\frac{1}{2}H_y(s,0)
\biggr\}\,ds,
\]
integrating $B^n$ and $B$ by parts we obtain
\begin{eqnarray*}
c N^n_1 &=& \frac{c}{\sigma}W_1
\bigl(H^n(1, \sigma W_1)-H(1, \sigma W_1)
\bigr)-\frac{c}{\sigma^2}\int_0^{\sigma W_1} \bigl\{
H^n(1,y)-H(1,y) \bigr\}\,dy
\\
&&{}+ c\int_0^1 \bigl\{H_y(s,\sigma
W_s) -H^n_y(s,\sigma W_s)
\bigr\} \,ds- \frac{c}{2}\int_0^1 \bigl
\{H_y(s,0) - H^n_y(s,0) \bigr\}\,ds.
\end{eqnarray*}
As $H^n$ are uniformly bounded and $H^n_y \leq C \frac{1}{\sqrt{1-t}}$
for $t <1$, if we can show that $H^n(1, y) \rar H(1,y)$ and $H^n_y(t,
y) \rar H_y(t,y)$ for all $y \in\bbR$ and $t \in[0,1)$, the above
will immediately imply that $N_1^n$ converges to $0$, $\bbQ$-a.s.
Moreover, it will also imply convergence in $L^p(\bbQ)$ for all $p \in
[1, \infty)$ in view of the bound obtained in (\ref{e:Wdominates}). In
particular, we will have
%
\begin{equation}
\label{e:QVlimit} \lim_{n \rar\infty} \bbE^{\bbQ} \biggl[\int
_0^1 W_s^2 \bigl
\{H^n_y(s, \sigma W_s)-H_y(s,
\sigma W_s) \bigr\}^2\,ds \biggr]=0.
\end{equation}
Thus,
%
\begin{equation}\qquad
\label{e:QVlimit1} \lim_{n \rar\infty} \int_0^1
W_s^2 \bigl\{H^n_y(s, \sigma
W_s)-H_y(s, \sigma W_s) \bigr
\}^2\,ds =0\qquad \mbox{in $\bbQ$-probability}.
\end{equation}
Moreover,
\begin{eqnarray*}
&&\int_0^1 W_s^2
\bigl\llvert H^n_y(s, \sigma W_s)-H_y(s,
\sigma W_s)\bigr\rrvert H_y(s, \sigma W_s)
\,ds
\\
&&\qquad\leq\int_0^1 W_s^2
\bigl\{H^n_y(s, \sigma W_s)-H_y(s,
\sigma W_s) \bigr\}^2\,ds\int_0^1
W_s^2 \bigl(H_y(s, \sigma W_s)
\bigr)^2\,ds,
\end{eqnarray*}
which in turn [since due to (\ref{e:Wdominates}), $\int_0^1 W_s^2
(H_y(s, \sigma W_s) )^2\,ds<\infty$ $\bbQ$-a.s.] implies
%
\begin{eqnarray}
\label{e:QVlimit2} \lim_{n \rar\infty} \int_0^1
W_s^2 \bigl\{H^n_y(s, \sigma
W_s)-H_y(s, \sigma W_s) \bigr
\}H_y(s,\sigma W_s)\,ds =0
\nonumber
\\[-8pt]
\\[-8pt]
\eqntext{\mbox{in $
\bbQ$-probability}.}
\end{eqnarray}
Combining (\ref{e:QVlimit1}) and (\ref{e:QVlimit2}), we can deduce that
\[
\lim_{n \rar\infty} \int_0^1
W_s^2 \bigl(H^n_y(s, \sigma
W_s) \bigr)^2\,ds=\int_0^1
W_s^2 \bigl(H_y(s, \sigma W_s)
\bigr)^2\,ds\qquad \mbox{in $\bbQ$-probability}.
\]
Together with $N^n_1 \rar0$, $\bbQ$-a.s. the above implies $M^n_1
\rar
M_1$, $\bbQ$-probability.

Thus, it remains to show that $H^n(1, y) \rar H(1,y)$ and $H^n_y(t, y)
\rar H_y(t,y)$ for all $y \in\bbR$ and $t \in[0,1)$. Indeed, $\lim_{n
\rar\infty}H^n(1,y)= \lim_{n \rar\infty}f(\Phi
^{-1}(P_n(y)))=f(\Phi
^{-1}(P(y)))$ in view of the continuity of $f\circ\Phi^{-1}$ on
$(0,1)$ and the fact that the sequence $(P_n(y))$ converges to a limit
$P(y)\in(0,1)$ for any $y \in\bbR$ due to the definition of $D$.

Next, observe that for $t<1$
\begin{eqnarray*}
&&\bigl|H^n_y (t,y)-H_y (t,y)\bigr| \\
&&\qquad\leq\int
_{\bbR} \bigl\llvert H^n(1,z+y)-H(1,z+y) \bigr
\rrvert\frac{|z|}{\sigma^2(1-t)}q \bigl(\sigma^2(1-t), z \bigr) \,dz.
\end{eqnarray*}
As $H^n$ and $H$ are bounded by $\|f\|_{\infty}$, the convergence to
$0$ follows from the dominated convergence theorem and that $H^n(1, y)
\rar H(1,y)$ as $n \rar\infty$.

Thus, we have verified that $T$ is continuous operator, $D$ is a closed
and convex subset of a Banach space and $TD$ is precompact. Therefore,
by Schauder's fixed-point theorem, $T$ has a fixed point $P$, that is, $TP=P$.
For this $P$, define $H$ as the solution to (\ref{eq:BVPH}) and $Y$ as
the corresponding unique solution to (\ref{eq:NCY}). Then $(H,Y)$ is
the solution to the system of equations (\ref{eq:NCH})--(\ref{eq:NCF}).

To complete the proof of the theorem, we need to show that the solution
to (\ref{eq:NCY}) has a transition density with the required smoothness
and positivity properties. This follows from the Lemma~\ref{l:density}
once we
observe that $h(z):=f(\Phi^{-1}(P(z)))$ satisfies the required
conditions. It is obvious that $h$ is bounded (since $f$ is),
nonconstant and nondecreasing (as $f$, $\Phi$ and $P$ are). Moreover,
$h'(y)=f'(\Phi^{-1}(P(y)))(\Phi^{-1})'(P(y))P'(y)$, is well defined for
all $y\in\bbR$ as $P(y)\in(0,1)$ for all $P\in D$ and $y\in\bbR$ and,
therefore, $h$ is absolutely continuous. Let $K\subset\bbR$ be a
compact, then since $P$ is continuous, $P(z)\in K_1$ for all $z\in K$,
where $K_1\subset(0,1)$ is also a compact. As $\Phi^{-1}\in
C^1((0,1))$, this implies that $(\Phi^{-1})'(P(y))$ is bounded for all
$y\in K$. Similarly, $f'(\Phi^{-1}(P(y)))$ is bounded for all $y\in K$.
As boundedness of $P'$ follows from the fact that $P\in D$, this yields
that $h'$ is bounded on compacts and, therefore, satisfies the conditions
of Lemma~\ref{l:density}.
\end{pf*}

\section{Construction of the equilibrium}\label{s:eqm}
Suppose $H$ is the function determined in Theorem~\ref{th:fp}. As
briefly discussed in Section~\ref{s:eq}, if we can identify an
admissible strategy $X$ such that: (i) $\what{\alpha}$ is given by
$-\frac{\sigma^2 \rho}{2N}Y_t H_y(t,Y_t)$, and (ii) $X_1$
satisfies (\ref{eq:iopt}), then $(H,X)$ will be a candidate equilibrium
in view of Proposition~\ref{p:ioptimality} once we show that $U(G)$ is
a true $\cF^M$-martingale. The following theorem gives such an $X$.
%

\begin{theorem} \label{t:insider} Let $H$ and $p$ be the functions
defined in Theorem~\ref{th:fp}. Then there exists a unique process
$(Y_t)_{t \in[0,1)}$ which solves
%
\begin{eqnarray}
\label{eq:sdeY} dY_t=\sigma\,dB_t + \biggl\{-
\frac{\sigma^2 \rho}{2N} Y_t H_y(t,Y_t) +
\sigma^2 \frac{p_y}{p}\bigl(t,Y_t;1,
H^{-1}(1,V)\bigr) \biggr\}\,dt,
\nonumber
\\[-8pt]
\\[-8pt]
\eqntext{t \in[0,T],}
\end{eqnarray}
for all $T<1$. Moreover, $Y$ is a $(\bbP^v,\cF^I)$-semimartingale with
$\bbP^v(\lim_{t\rar1}Y_t=H^{-1}(1,V))=1$ for every $v \in f(\bbR)$ and
\[
dY_t=\sigma\,dB^Y_t -\frac{\sigma^2 \rho}{2N}
Y_t H_y(t,Y_t)\,dt, \qquad t \in[0,1]
\]
under $\cF^M$.
\end{theorem}

\begin{pf}
We will first show that there exists a unique weak solution to (\ref
{eq:sdeY}) on $[0,T]$ for any $T<1$. Then Proposition IX.3.2 in \cite
{ry} will imply the uniqueness of strong solutions since if $Y^1$ and
$Y^2$ are two strong solutions, then $Y^1-Y^2$ satisfies
\[
Y^1_t-Y^2_t= \int
_0^t b \bigl(s, Y^1_s,
Y^2_s, V \bigr)\,ds
\]
for some deterministic function $b$ and, therefore, its local time
process at level $0$ is identically $0$. The strong uniqueness combined
with a weak solution will lead to the existence of a unique strong
solution by a result due to Yamada and Watanabe (see Corollary~5.3.23
in \cite{ks}). To show the existence of a weak solution, fix $T<1$ and
let $N_t:=p(t,\zeta_t; 1, H^{-1}(1,v))$ for $t \leq T$ where $v \in
\bbR
$ and $\zeta$ is the unique strong solution of
\[
d\zeta_t =\sigma\,d\beta_t -\frac{\sigma^2 \rho}{2N}
\zeta_t H_y(t,\zeta_t)\,dt
\]
on $[0,1]$ under a probability measure $\tilde{\bbP}$, where $\beta$ is
a $\tilde{\bbP}$-Brownian motion as established in Theorem~\ref{th:fp}.
The same theorem also gives $p$ as the transition density of $\zeta$.
Then $(N_t)_{t \in[0,T]}$ is a strictly positive and bounded
martingale with respect to the natural filtration of $\zeta$ as a
consequence of It\^{o} formula and the estimates on $p$ obtained in (\ref
{eq:ryzbd}). Thus, $\frac{N_T}{N_0}$ has expectation $1$ under $\tilde
{\bbP}$ and defines an equivalent change of measure on the $\sigma
$-algebra $\cF^{\zeta}_T$. Since
\[
dN_t=\sigma N_t \frac{p_y}{p} \bigl(t,
\zeta_t;1, H^{-1}(1,v) \bigr)\,d\beta_t,
\]
then it follows from Girsanov's theorem that under the new measure,
$\bbQ^T$,
\[
d\zeta_t =\sigma\,dW_t + \biggl\{-\frac{\sigma^2 \rho}{2N}
\zeta_t H_y(t,\zeta_t) + \sigma^2
\frac{p_y}{p} \bigl(t,\zeta_t;1, H^{-1}(1,v) \bigr)
\biggr\} \,dt
\]
for some $\bbQ^T$-Brownian motion. Thus, $\zeta$, as a solution of the
above under $\bbQ^T$ is a weak solution of (\ref{eq:sdeY}) on $[0,T]$.
Moreover, the weak uniqueness holds since the distribution of $\zeta$
under $\bbQ^T$ has a one-to-one correspondence with the distribution of
$\zeta$ under the original measure via the change of measure
martingale $p(t,\zeta_t; 1, H^{-1}(1,v))$. More precisely, for any
bounded function $F$ and points $0 = t_0 < \cdots<t_n = T$,
%
%
\begin{eqnarray}
\label{eq:Ylaw} &&\hspace*{-4pt}\bbE^{\bbQ^T} \bigl[F(\zeta_{t_1}, \ldots,
\zeta_{t_n}) \bigr]\nonumber\\
&&\hspace*{-6pt}\qquad= \bbE^{\tilde{\bbP
}} \biggl[F(\zeta_{t_1},
\ldots, \zeta_{t_n})\frac{p(T,\zeta_T; 1,
H^{-1}(1,v))}{p(0,0; 1, H^{-1}(1,v))} \biggr]
\nonumber
\\[-8pt]
\\[-8pt]
\nonumber
&&\hspace*{-6pt}\qquad=\int_{\bbR}\cdots\int_{\bbR}F(y_1,
\ldots, y_n)\\
&&\hspace*{-12pt}\qquad\quad{}\times\frac{p(0,0;t_1,y_1)\cdots
p(t_{n-1},y_{n-1};t_n,y_n)p(T,y_n;1,H^{-1}(1,v))}{p(0,
0;1,H^{-1}(1,v))}\,dy_1\cdots
\,dy_n.
\nonumber
\end{eqnarray}

Hence, we conclude the existence of a unique strong solution, $Y^T$, of
(\ref{eq:sdeY}) over the interval $[0,T]$ under $\bbP^v$. Define $Y$ by
$Y_t = Y^T_t \chf_{t \leq T}$ and observe that due to the uniqueness of
strong solutions $Y$ is well defined and is the unique process that
solves (\ref{eq:sdeY}).

Next, we want to extend the process $Y$ to time-1 by considering its
limit. This limit exists in view of Theorem~2.2 in \cite{smb}. Note
that Assumption~2.2 of \cite{smb} is satisfied since $(t,y) \mapsto
p(t,y;u,z)$ is $C^{1,2}$ on $[0,u)\times\bbR$ and
$p(t,y;u,z)=q(\sigma
^2(u-t), z-y) r(t,y;u,z)$, where $q$ is the transition density of
standard Brownian motion and $r$ is a strictly positive function with
exponential bounds given by~(\ref{eq:ryzbd}), which in particular
implies $p$ generates a time inhomogeneous Feller semigroup. Moreover,
the same result also yields $\bbP^v(\lim_{t\rar1}Y_t=H^{-1}(1,V))=1$.

To show the semimartingale property of $Y$ let $z=H^{-1}(1,v)$ and
recall from (\ref{e:rdef}) that
\begin{eqnarray*}
&&r(t,y;1,z)\\
&&\qquad=\bbE^{\bbQ^{y \rar z}_{t \rar1}} \biggl[\exp \biggl(-\frac{\rho
}{2N}\int
_t^1 X_s H_y(s,X_s)
\,dX_s -\frac{\sigma^2\rho^2}{8 N^2}\int_t^1
X_s^2 H^2_y(s, X_s)
\,ds \biggr) \biggr].
\end{eqnarray*}
In view of the Markov property of Brownian bridges, we have
\begin{eqnarray*}
&&r(t,X_t;1,z)\\
&&\qquad=\bbE^{\bbQ^{0 \rar z}_{0 \rar1}} \biggl[\exp \biggl(-
\frac
{\rho}{2N}\int_t^1 X_s
H_y(s,X_s)\,dX_s\\
&&\qquad\quad{} -\frac{\sigma^2\rho^2}{8 N^2}\int
_t^1 X_s^2
H^2_y(s, X_s) \,ds \biggr) \Big|
\cF^X_t \biggr],
\end{eqnarray*}
where $(\cF^X)$ is the usual augmentation of the natural filtration of
$X$ since both $\int_t^1 X_s H_y(s,X_s)\,dX_s$ and $\int_t^1 X_s^2
H^2_y(s, X_s) \,ds$ are measurable with respect to $\sigma(X_u; u\in[t,1])$.

Therefore,
\[
L_t:=r(t,X_t;1,z)\exp \biggl(-\frac{\rho}{2N}\int
_0^t X_s H_y(s,X_s)
\,dX_s -\frac{\sigma^2\rho^2}{8 N^2}\int_0^t
X_s^2 H^2_y(s, X_s)
\,ds \biggr)
\]
is a $\bbQ^{0 \rar z}_{0 \rar1}$-martingale. Moreover, it is strictly
positive due to (\ref{e:rispos}). Thus, we can define $\bbP^{0,z}$ on
$\cF^X_1$ via $\frac{d\bbP^{0,z}}{d\bbQ^{0 \rar z}_{0 \rar1}}=L_1$.

Recall that under $\bbQ^{0 \rar z}_{0 \rar1}$ $X$ solves the
following SDE:
\[
X_t= \sigma W^z_t + \int
_0^t\frac{z-X_s}{1-s}\,ds,
\]
where $W^z$ is a $\bbQ^{0\rar z}_{0 \rar1}$-Brownian motion. Thus, a
straightforward application of Girsanov's theorem yields that $X$
solves (\ref{eq:sdeY}) once $B$ is replaced with the $\bbP
^{0,z}$-Brownian motion defined by the Girsanov transform since
\[
\frac{dL_t}{L_t}= \biggl(\frac{r_x(t,X_t;1,z)}{r(t,X_t;1,z)}-\frac
{\sigma
\rho}{2 N}X_tH_y(t,X_t)
\biggr)\,dW^z_t.
\]
Since semimartingale property is preserved under equivalent changes of
measure and the strong uniqueness holds for the solutions of (\ref
{eq:sdeY}), we obtain the desired semimartingale property of its unique
solution.

Having shown the semimartingale property it remains to demonstrate the
claimed representation of $Y$ under $\cF^M$. Suppose that $\xi$ is a
solution of
%
\begin{equation}
\label{eq:xi} \xi_t=\sigma\beta_t -\int
_0^t \frac{\sigma^2 \rho}{2N} \xi_s
H_y(s,\xi _s ) \,ds.
\end{equation}
Then $\xi$ has the transition density $p$. If one enlarges the
filtration of $\xi$ with $\xi_1$, then under the enlarged filtration
$\xi$ has the following decomposition:
%
\begin{eqnarray}
\label{eq:xie} d\xi_t=\sigma\,dW_t + \biggl\{-
\frac{\sigma^2 \rho}{2N} \xi_t H_y(t,\xi _t) +
\sigma^2 \frac{p_y}{p}(t,\xi_t;1, \xi_1)
\biggr\}\,dt,
\nonumber
\\[-8pt]
\\[-8pt]
\eqntext{t \in[0,1),}
\end{eqnarray}
where $(W_t)_{t \in[0,1)}$ is a Brownian motion in the enlarged
filtration independent of $\xi_1$ (see Theorem~1.6 in \cite{my}).

On the other hand, since $\bbE[\xi_t|\cF^{\xi}_t]=\xi_t$ for
$t<1$, we
must have
\begin{eqnarray*}
\sigma\beta_t -\int_0^t
\frac{\sigma^2 \rho}{2N} \xi_s H_y(s,\xi _s )
\,ds&=&\bbE\bigl[\sigma W_t|\cF^{\xi}_t\bigr]+
\bbE\biggl[ \int_0^t \sigma^2
\frac
{p_y}{p}(s,\xi_s;1, \xi_1) \,ds \Big|
\cF^{\xi}_t \biggr]
\\
&&{} -\int_0^t \frac{\sigma^2 \rho}{2N}
\xi_s H_y(s,\xi_s) \,ds.
\end{eqnarray*}
Since the projection of a martingale onto a smaller filtration is still
a martingale, from the above equation we conclude that $\bbE[ \int_0^t \sigma^2 \frac{p_y}{p}(s,\xi_s;1, \xi_1) \,ds |\cF^{\xi}_t
]$ is an $\cF^{\xi}$-martingale which is equivalent to
%
\begin{equation}
\label{eq:xim} \bbE\biggl[ \int_t^u
\sigma^2 \frac{p_y}{p}(s,\xi_s;1, \xi_1)
\,ds \Big|\cF^{\xi}_t \biggr]=0 \qquad\forall u\in[t,1).
\end{equation}

Observe that by Theorem~\ref{th:fp} the distribution of $\xi_1$ and
that of $H^{-1}(1,V)$ coincide. Since we have established the
uniqueness in law of solutions to (\ref{eq:sdeY}) and $V$ is
independent of $B$, we can conclude that the processes $Y$ and $\xi
$ have the same distribution. Thus, from (\ref{eq:xim}) it follows that
for $u<1$
\begin{eqnarray*}
\bbE \biggl[ \int_t^u \sigma^2
\frac{p_y}{p} \bigl(s,Y_s;1, H^{-1}(1,V) \bigr) \,ds \Big|
\cF^{Y}_t \biggr]&=&\bbE \biggl[ \int_t^u
\sigma^2 \frac
{p_y}{p}(s,Y_s;1, Y_1)
\,ds \Big|\cF^{Y}_t \biggr]\\
&=&0.
\end{eqnarray*}
The above implies that $\bbE[ \int_0^t \sigma^2 \frac
{p_y}{p}(s,Y_s;1, H^{-1}(1,V)) \,ds |\cF^{Y}_t ]$ is an $\cF
^Y$-martingale. Therefore, $Y$ has the following decomposition with
respect to $\cF^Y$:
\[
Y_t=M_t- \int_0^t
\frac{\sigma^2 \rho}{2N} Y_s H_y(s,Y_s) \,ds, \qquad t
\in[0,1),
\]
where $M$ is an $\cF^Y$-martingale. On the other hand, $[M,M]_t=[Y,
Y]_t=\sigma^2 t$. Thus, by L\'evy's characterisation, $M_t= \sigma
B^Y_t$. Note that $(B^Y_t)_{t \in[0,1)}$ is a uniformly integrable
martingale; thus, we can define $B^Y_1=\lim_{t \rar1} B^Y_t$ so that
$(B^Y_t)_{t \in[0,1]}$ is a Brownian motion. This establishes the
desired decomposition on $[0,1]$ as $Y$ converges to a finite limit as
$t \rar1$.
\end{pf}

Theorem~\ref{t:insider} in conjunction with Proposition~\ref
{p:ioptimality} establish the existence of an equilibrium in our model.
%

\begin{theorem} \label{t:eq} Let $H^{\ast}$ and $p$ be the functions
defined in Theorem~\ref{th:fp}, and
\[
X^{\ast}_t=\int_0^t
\biggl\{-\frac{\sigma^2 \rho}{2N} Y^{\ast}_s H^{\ast
}_y
\bigl(s,Y^{\ast}_s \bigr) +\sigma^2
\frac{p_y}{p} \bigl(s,Y^{\ast}_s;1, H^{\ast
{-1}}(1,V)
\bigr) \biggr\} \,ds.
\]
Then $(H^{\ast}, X^{\ast})$ is an equilibrium.

Moreover, under $\cF^M$ the equilibrium demand evolves as
\[
Y^{\ast}_t=\sigma B^Y_t -
\frac{\sigma^2 \rho}{2N} \int_0^t
Y^{\ast}_s H^{\ast}_y
\bigl(s,Y^{\ast}_s \bigr) \,ds.
\]
\end{theorem}

\begin{pf} Note that $H^{\ast}$ is a bounded function being a
solution of heat equation with a bounded terminal condition. Thus,
conditions (\ref{a:Hintegrability}) and (\ref{a:Xintegrability}) are
automatically satisfied. Moreover, Theorem~\ref{t:insider} yields that
$X^*$ is a $(P^v,\cF^I)$-semimartingale and $\bbP^v(H^*(1,Y^*_1)=v)=1$.
Thus, Proposition~\ref{p:ioptimality} yields that $X$ is admissible and
optimal strategy for the insider given $H^*$.

Thus, it remains to verify the zero-utility gain condition of the
market makers, that is, to prove that $U(G)$ is an $\cF^M$-martingale.
Recall from Theorem~\ref{t:insider} that with this choice of $X^{\ast
}$, $Y^{\ast}$ solves
%
\begin{equation}
\label{eq:YP} dY_t=\sigma\,dB^Y_t -
\frac{\sigma^2 \rho}{2N} Y_t H^{\ast
}_y(t,Y_t)
\,dt.
\end{equation}
Thus, It\^{o}'s formula together with the conditions on $H^{\ast}$ yields
\[
U(G_t)= -\exp \biggl(\frac{\sigma\rho}{N}\int_0^t
Y^{\ast}_s H^{\ast
}_y
\bigl(s,Y^{\ast}_s \bigr) \,dB^Y_s-
\frac{\sigma^2 \rho^2}{2 N^2}\int_0^t
\bigl(Y^{\ast}_s H^{\ast}_y
\bigl(s,Y^{\ast}_s \bigr) \bigr)^2 \,ds \biggr).
\]
Clearly, $-U(G)$ is an exponential local martingale.

Next, observe in view of the absolute continuity relationship between
the laws of $Y$ and $\sigma W$ as established in Lemma~\ref{l:girsanov}
$(\frac{1}{M^*_t})_{t \in[0,1]}$ is a strictly positive $\bbP
$-martingale, where
\[
M^*_t= \exp \biggl(-\frac{\rho}{2N} \int_0^t
Y^*_s H_y \bigl(s,Y^*_s \bigr)
\,dY_s- \frac
{\sigma^2 \rho^2}{8 N^2}\int_0^t
\bigl(Y^*_s H_y \bigl(s,Y^*_s \bigr)
\bigr)^2 \,ds \biggr).
\]
Therefore, if we define an equivalent measure, $\bbQ$, on $\cF^{Y^*}_1$
by $\frac{d\bbP}{d\bbQ}=\frac{1}{M^*_1}$, then $W^Y:=\frac
{Y^*}{\sigma
}$ is a $\bbQ$-Brownian motion. Consequently, $U(G)$ is a $\bbP
$-martingale if and only if $U(G)M^*$ is $\bbQ$-martingale.

On the other hand, a straightforward application of integration by
parts formula yields
\[
d \bigl(-U(G_t)M^*_t \bigr)=-U(G_t)M^*_t
\frac{\sigma\rho}{2N}Y_t H_y(t, Y_t)
\,dW^Y_t,
\]
that is, $-U(G)M^*$ is the stochastic exponential of $\int_0^{\cdot
}\frac{\sigma\rho}{2N}Y_t H_y(t, Y_t)\,dW^Y_t$. Moreover,
\[
\biggl\llvert\int_0^{1}Y_t
H_y(t, Y_t)\,dW^Y_t \biggr\rrvert
\leq K \bigl(1+\bigl |W^Y_1\bigr| \bigr)
\]
by (\ref{e:Wdominates}). Since $|W^Y_1|$ has all exponential moments,
we conclude that $-U(G)M^*$ is a $\bbQ$-martingale using Kazamaki's
criterion (see, e.g., Theorem III.44 in \cite{P}).
\end{pf}

The above theorem shows that the equilibrium demand process has a drift
in its own filtration. This is in contrast with the other possible
generalisations found in the literature (for the change in the pattern
of private information arrival see~\cite{BP}, for a risk averse insider
see \cite{baruch} and for competition among insiders see \cite{bcw}) of
the original models of \cite{Kyle} and \cite{B} lead to equilibria with
total demand being a martingale in its own filtration.

Moreover, as $H_y>0$ the equilibrium total demand process is mean
reverting. This suggests a theoretical explanation for the emergence of
mean reversion in the specialists' inventories, which has strong
empirical support (see, e.g., \cite{BR05,HNV98,HS97}
and \cite{MS93}): The mean reversion appears as a result of the
insider's reaction to the market maker's demand for risk sharing. The
speed of mean reversion is not constant and depends on the market
makers' level of effective risk aversion, $\frac{\rho}{N}$, as well as
the level of informational asymmetry, $\sigma$, in a nontrivial way
due to the definition of $H_y$. This theoretical implication is in line
with the empirical findings of \cite{HNV98} who observe that the speed
of mean reversion depends on the inventory levels of the market makers
in the London Stock Exchange.

Closely related to the observation that the total order has a drift, is
the fact that the equilibrium price is no longer a martingale under the
physical measure. Moreover, $Y^{\ast}$ and, therefore, $H^{\ast
}(t,Y^{\ast}_t)$ are mean-reverting processes. This mean-reversion
property of $Y^{\ast}$ also entails that Kyle's conclusion of constant
market depth, which is the order size necessary to move the prices by
one unit, does not hold in this model. Indeed,
\[
dH^{\ast}_y \bigl(t,Y^{\ast}_t
\bigr)=H^{\ast}_{yy} \bigl(t,Y^{\ast}_t
\bigr) \sigma\,dB^Y_t -\frac{\sigma^2 \rho}{2 N}
Y^{\ast}_t H^{\ast}_y
\bigl(t,Y^{\ast}_t \bigr)H^{\ast
}_{yy}
\bigl(t,Y^{\ast}_t \bigr) \,dt
\]
implies that $H^{\ast}_y(t,Y^{\ast}_t)$ is not a martingale since
$H^{\ast}$ is not linear. In particular, if $H^{\ast}$ is such that
$H^{\ast}_{yy}(1,y) = -H^{\ast}_{yy}(1,-y) $ with $H^{\ast
}_{yy}(1,y)\leq0$ for $y\geq0$, then $yH^{\ast}_{yy}(t,y)\leq0$, and
thus $H^{\ast}_y(t,Y^{\ast}_t)$ is a submartingale. Consequently,
$\bbE
[H^{\ast}_y(t,Y^{\ast}_t)]$ has an upward slope, that is, the
executions costs increase in time in our model. This is consistent with
the empirical findings of U-shaped patterns of execution costs on NYSE
(see \cite{mrr}).
\section{Conclusion and further remarks}\label{s:conclusion}
We have solved a long-standing open problem first posed by Subrahmanyam
in \cite{sub} of existence of an equilibrium in a financial market with
asymmetrically informed traders and risk averse market makers in a
continuous-time version of a model first introduced by Kyle \cite
{Kyle}. The equilibrium turns out to be the solution of a nonstandard
FBSDE. We have solved this FBSDE by transforming it into a
forward--backward system of stochastic and partial differential
equations and employing a novel application of Schauder's fixed-point theorem.

Consistent with the empirical studies on the inventories of market
makers we find that the risk aversion of market makers causes mean
reversion in the equilibrium total demand (i.e., collective inventory
of the market makers). This implies that the informed trader's strategy
ceases to be \emph{inconspicuous} and, therefore, provides the first
example of an equilibrium in a Kyle-type model which does not satisfy
inconspicuousness condition.
The driving force behind this result is that the risk aversion of
market makers makes them unwilling to bear risk. Instead of paying the
extra compensation for the inventory risk, the informed trader chooses
to absorb a part of large fluctuations in the market makers'
inventories, that is, participates in a risk sharing.

We also show that the sensitivity of prices to the total order, which
is the reciprocal of the market depth, can be a sub-martingale for
certain model parameters. This implies that the execution costs are, on
average, increasing toward the end of a trading period, which is
consistent with the empirical results obtained in \cite{mrr}.

Whereas, for general set of the parameters, the reciprocal of the
market depth is not a sub-martingale, it is not a martingale either.
This theoretical conclusion is in discord with the results obtained in
\cite{Kyle}, as well as in \cite{baruch}, who studies the effect of
risk averse insider on the equilibrium, and in \cite{BP}, who extend
Kyle's model to the case when the informed trader receives a
fluctuating signal over time. In fact, Kyle in \cite{Kyle} made a
conjecture that:

\begin{quote}
[$\cdots$] neither increasing nor decreasing depth is consistent with
behavior by the informed trader which is ``stable'' enough to sustain
an equilibrium. If depth ever increases, the insider wants to
destabilize prices (before the increase in depth) to generate unbounded
profits. If depth ever decreases, the insider wants to incorporate all
of his private information into the price immediately.
\end{quote}

Thus, the results obtained from our model demonstrate that the
necessity of risk sharing between the informed trader and the market
makers makes exploitation of systematic movements in market depth
unprofitable for the informed trader. Indeed, if the trader attempts to
follow the strategy outlined by Kyle, that is, acquiring a large
position when depth is lower in order to liquidate at more favourable
price when depth is higher, she would be moving the total order away
from its mean, leaving the market makers exposed to the risk of large
orders. Violation of risk sharing would cause the market makers to
adjust the prices eliminating favourable liquidation opportunities for
the informed trader. Thus, contrary to Kyle, such a strategy does not
lead to unbounded profits.

Moreover, the appearance of systematic changes in market depth as a
result of market makers' risk aversion demonstrates that competition of
the informed traders, as in \cite{bcw}, is not the only possible
mechanism that can make the strategy proposed by Kyle unprofitable,
thus leading to a drift in the reciprocal of the market depth.

These observations show that a mere introduction of risk averse market
makers to the setting of \cite{Kyle} changes the equilibrium outcome
fundamentally.

%





\printaddresses
\end{document}